\documentclass[10pt]{article} 

\usepackage{amsmath, amsfonts, amssymb}
\usepackage{color}

\input{liemacs10.sty} 
\newcommand{\coz}{\mathop{{\rm co}_\fz}\nolimits}
\newcommand{\co}{\mathop{{\rm co}}\nolimits}

\newcommand{\be}{\mathbf{e}}
\newcommand{\bw}{\mathbf{w}}
\newcommand{\bv}{\mathbf{v}}

\renewcommand{\phi}{\varphi}

\usepackage{hyperref}
\usepackage{pdfpages}
\usepackage{pgfplots}

\renewcommand\mlabel{\label}

\begin{document}

\title{Elements in pointed invariant cones in Lie algebras\\ 
and corresponding affine pairs} 
\author{Karl-Hermann Neeb and Daniel Oeh
\thanks{This research has been supported
  by DFG-grant NE 413/10-1}
\thanks{Department of Mathematics,
 Friedrich-Alexander-University Erlangen--Nuremberg,
Cauerstrasse 11,\break 91058 Erlangen, Germany, 
 neeb@math.fau.de; oehd@math.fau.de}}
\maketitle


\begin{abstract}
In this note we study in a finite dimensional Lie algebra 
$\g$ the set of all those elements $x$ for which the 
closed convex hull  of the adjoint orbit contains 
no affine lines; this contains in particular elements whose adjoint 
orbits generates a pointed convex cone~$C_x$. 
Assuming that $\g$ is admissible, i.e., contains a generating invariant 
convex subset not containing affine lines, we obtain a natural characterization  
of such elements, also for non-reductive Lie algebras.

Motivated by the concept of standard (Borchers) pairs in 
QFT, we also study pairs $(x,h)$ of Lie algebra elements 
satisfying $[h,x]=x$ for which $C_x$ pointed. 
Given $x$, we show that such elements $h$ 
can be constructed in such a way that $\ad h$ defines a $5$-grading, 
and characterize the cases where we even get a $3$-grading.   

\end{abstract}

\section{Introduction} 
\mlabel{sec:1}

Convexity properties of adjoint orbits 
$\cO_x = \Inn(\g)x$ in a finite dimensional real Lie algebra, 
where $\Inn(\g) = \la e^{\ad \g} \ra$ is  the group of inner automorphisms, 
play a role in many contexts. Most directly, they appear in the theory 
of invariant convex cones. For instance, if $U \: G \to \U(\cH)$ 
is a unitary representation and 
$\partial U(x)$ denotes the infinitesimal generator of the 
unitary one-parameter group $(U(\exp t x))_{t \in \R}$, then the 
{\it positive cone of $U$} 
\begin{equation}
  \label{eq:CU}
 C_U := \{ x \in \g \: -i \partial U(x) \geq 0\}  
\end{equation}
is a closed convex invariant cone in $\g$ which is pointed 
(contains no affine lines) if and only if $\ker(U)$ is discrete. 
In the literature pointed generating invariant cones have been studied from the 
perspective of their interior: If $W \subeq \g$ is pointed and generating,
then its interior consists of elliptic elements 
($\ad x$ is semisimple with imaginary spectrum) and $W$ is determined 
by its intersection with a compactly 
embedded Cartan subalgebra (\cite{HHL89, Ne00}). Unfortunately, 
this theory provides not much information on the non-elliptic elements 
in the boundary of~$W$. A notable exception 
is \cite{HNO94} which, for a simple hermitian Lie algebra, 
 provides a classification of all nilpotent adjoint orbits in an invariant cone. 

In the present paper we address adjoint orbits in invariant cones. 
For more precise formulations, we introduce some notation. 
For $x \in \g$, we write 
\begin{itemize}
\item $\co(x) := \oline{\conv(\cO_x)},$ for the  closed convex hull of $\cO_x$, and 
\item  $C_x :=  \cone(\cO_x) = \oline{\R_+\conv(\cO_x)}$ %
for the closed convex cone generated by~$x$. 
\end{itemize}
We call a closed convex subset $C$ of a real linear space $V$
 {\it pointed} if it contains no 
non-trivial affine lines and {\it generating} if $\Spann C = V$. 
We study the subsets 
\begin{equation}
  \label{eq:gc1}
  \g_{\rm co} := \{ x \in \g \: \co(x)\ \mbox{pointed} \} 
\supeq   \g_{\rm c} := \{ x \in \g \: C_x\ \mbox{pointed} \}
\end{equation}
and characterize the elements in this set in terms of explicitly 
available data. 

If $\g$ is a simple Lie algebra, then we have to distinguish three 
cases. If $\g$ is compact, then all sets $\co(x)$ are compact, hence 
pointed, so that $\g_{\rm co} = \g$. As every convex cone 
invariant under a compact group contains a fixed point in its interior, 
we have $\g_{\rm c} = \{0\}$ for compact Lie algebras. 
If $\g$ is non-compact, then $\g_{\rm c} \not= \{0\}$ is equivalent to 
$\g$ being hermitian, i.e., maximal compactly embedded subalgebras 
have non-trivial center. In this 
case $\g_{\rm c} = \g_{\rm co}$ is a double cone 
(cf.\ Lemma~\ref{lem:cN}, Kostant--Vinberg Theorem). If $\g$ is neither compact nor hermitian,
then $\g_{\rm c} = \g_{\rm co}= \{0\}$. 

For a direct sum $\g = \g_1 \oplus \g_2$ we have 
$\g_{\rm co} = \g_{1,{\rm co}} \times \g_{2,{\rm co}}$. 
This reduces for a reductive Lie algebra  
the determination of this set to the case of simple hermitian 
Lie algebras. However, the determination of the subset $\g_{\rm c}$ 
of $\g_{\rm co}$ is less obvious for reductive Lie algebras 
(Proposition~\ref{prop:3.14}). 

This discussion shows that, among the reductive Lie algebras 
only the {\it quasihermitian} ones 
(all simple ideals are either compact or hermitian), 
play a role in our context. Beyond reductive Lie algebras, 
the natural context for our investigation 
is the class of {\it admissible} Lie algebras, 
i.e., those containing an $\Inn(\g)$-invariant pointed generating 
closed convex subset $C$ (cf.~\cite[Def.~VII.3.2]{Ne00}).
Since every pointed invariant convex subset spans an admissible ideal, 
we shall assume throughout that $\g$ is admissible. 

It is of vital importance for our arguments, that 
admissible Lie algebras permit a powerful structure theory. 
Their coarse structure is given by 
$\g = \g(\fl,V,\fz,\beta) = \fz \oplus V \oplus \fl$, 
where $\fl$ is a reductive Lie algebra, 
$V$ an $\fl$-module, $\fz$ a vector space, and 
$\beta \: V \times V \to \fz$ an $\fl$-invariant skew-symmetric bilinear map; 
the Lie bracket on $\g$ is given by  
\begin{equation}
  \label{eq:VII.2.1a}
 [(z,v,x), (z',v',x')] 
= (\beta(v,v'), x.v' - x'.v, [x,x']) 
\in \fz \oplus V \oplus \fl. 
\end{equation}
For their fine structure, we use the existence of a 
compactly embedded Cartan subalgebra $\ft = \fz \oplus \ft_\fl$ 
and the corresponding root decomposition. 

The key observation underlying our analysis of 
the sets $\g_{\rm c}$ and $\g_{\rm co}$ in Section~\ref{sec:3}, 
is the Reduction Theorem~\ref{thm:conj1}, asserting 
that every adjoint orbit in $\g_{\rm co}$ intersects 
the reductive subalgebra ${\fz + \fl}$. 
We therefore take in Subsection~\ref{subsec:3.2} a closer 
look at the reductive case, where we provide in 
Propositions~\ref{prop:3.10} and \ref{prop:3.14} 
a complete description of the sets $\g_{\rm c}$ and $\g_{\rm co}$. 
The central result in Section~\ref{sec:3} is the 
Characterization Theorem~\ref{thm:2.12} that characterizes 
elements $x$ in $\g_{\rm c}$ and $\g_{\rm co}$ in terms of the 
closed convex hull $\co_\fz(x)$ of the 
$\fz$-valued Hamiltonian function 
\[ H_x^\fz(v) = 
p_\fz(e^{\ad v}x) = x_\fz + [v,x_V] +  \frac{1}{2} [v, [v,x_\fl]]
\quad \mbox{ for } \quad x  = x_\fz + x_V + x_\fl, v \in V. \] 
In particular, we show that $\co(x)$ is pointed if and only if 
$\co_\fz(x)$ is pointed and that $C_x$ is pointed if, in addition, 
$\co_\fz(x)$ generates a pointed cone whenever the $\fl$-component $x_\fl$ 
of $x$ is nilpotent. We also discuss to which extent the pointedness 
of $C_x$ implies the existence of a pointed generating invariant 
cone $W \subeq \g$ containing $x$; which is not always the case 
(Example~\ref{ex:counterex}). 

In Section~\ref{sec:4} we study {\it affine pairs} $(x,h)$ for a 
pointed invariant cone~$W \subeq \g$. These pairs are characterized 
by the relations 
\begin{equation}
  \label{eq:affpa}
 x \in W \quad \mbox{ and } \quad [h,x] = x.
\end{equation}
The interest in these pairs stems from their relevance 
in Algebraic Quantum Field Theory (AQFT), where they arise 
from unitary representations $(U,\cH)$ of a corresponding Lie group $G$ 
and their positive cones $W = C_U$ (see \eqref{eq:CU}). 
If $U$ extends to an antiunitary representation 
of $G \rtimes \{\id_G, \tau_G\}$, where 
$\tau_G \in \Aut(G)$ is an involution and the corresponding involution $\tau 
\in \Aut(\g)$ satisfies $\tau(h) = h$ and $\tau(x) = -x$, then 
we can associate to $h$ a so-called {\it standard subspace} 
\[ V := \Fix\big(U(\tau_G) e^{\pi i \cdot \partial U(h)}\big) \subeq \cH, \] 
and these subspace encode localization data in QFT 
\cite{BGL02}. In this context $(U_x,V)$ with 
$U_x(t) := U(\exp tx)$ is called a {\it Borchers pair} 
or a {\it standard pair} (\cite{Le15}), and one would like to understand all 
those pairs arising from a given unitary representation  
(see \cite{NO17} for more details). 
This leads naturally 
to the problem to describe and classify affine pairs. 
For an affine pair $(x,h)$ the element $x$ is nilpotent, hence 
in particular not elliptic if it is not central. 
As we know from \cite{Ne21}, the most important affine pairs 
are those for which $h$ in an {\it Euler element of $\g$}, 
i.e., $\ad h$ is diagonalizable with possible eigenvalues 
$\{-1,0,1\}$ and the Lie algebra $\g$ is generated by 
$h$ and the cones $C_U \cap \g_{\pm 1}(h)$
(see \cite{Oeh20a, Oeh20b} for related classification results). 
Considering representations with discrete kernel then leads to the 
situation, where the cone $C_U \subeq \g$ is pointed, so that 
$\g = \g_U \rtimes \R h$, and $\g_U = C_U-C_U$ is an ideal containing the 
pointed generating invariant cone $C_U$. 

This motivates our investigations in Section~\ref{sec:4}, 
where we start with a nilpotent element  $x \in W$, 
$W$ a pointed generating invariant cone, and then consider derivations 
$D$ on $\g$ satisfying $Dx = x$ and $e^{\R D}W = W$.  
Our first main result on affine pairs 
is the Existence Theorem~\ref{thm:4.2},  
asserting that, for any nilpotent element $x \in \g_{\rm co}$ there exists 
such a derivation $D$ with $\Spec(D) \subeq 
 \big\{ 0, \pm \shalf, \pm 1\big\}$. 
The second main result 
characterizes the existence of Euler derivations with 
this property, i.e., where we even have $\Spec(D) \subeq \{0,\pm 1\}$ 
(Theorem~\ref{thm:eulderexist}). \\

\nin {\bf Notation:} 
\begin{itemize}
\item For a Lie algebra $\g$, we write $\Inn(\g) = \la e^{\ad \g} \ra$ 
for the group of inner automorphisms. For a Lie subalgebra 
$\fh \subeq \g$, we write $\Inn_\g(\fh) \subeq \Inn(\g)$ for the 
subgroup generated by $e^{\ad \fh}$. ``Invariance'' of subsets of $\g$ always 
refers to the group $\Inn(\g)$. 
\item A subalgebra $\fk \subeq \g$ is said to be {\it compactly embedded} 
if the closure of $\Inn_\g(\fk)$ in $\Aut(\g)$ is compact. 
A {\it compactly embedded Cartan subalgebra} is a compactly embedded 
subalgebra which is also maximal abelian. 
\item For $A \in \End(V)$, we write $V_\lambda(A) := \ker(A - \lambda\id_V)$ 
for the eigenspaces. 
\item A closed convex cone in a finite dimensional real vector space is 
simply called a {\it cone}. We write $\cone(S) = \oline{\R_+ \conv(S)}$ for the 
closed convex cone generated by a subset $S$ 
and 
\[ S^\star := \{ \alpha \in V^* \: (\forall v \in S)\ \alpha(v) \geq 0\}\] 
for the {\it dual cone}. 
\end{itemize}

\tableofcontents

\section{Structure of admissible Lie algebras} 
\mlabel{sec:2}

In this  section we collect some relevant results on 
the structure of admissible Lie algebras. 

\begin{defn} Let $C$ be a closed convex subset of the finite dimensional 
real vector space~$V$. 
Then the {\it recession cone of $C$} is 
\[ \lim(C) := \{ x\in V \: C + x \subeq C\} 
= \Big\{ x \in V \: x = \lim_{n \to \infty} t_n c_n,
c_n \in C, t_n \to 0, t_n \geq 0 \Big\} \] 
(cf.\ \cite[Prop.~V.1.6]{Ne00}). 
The {\it edge of this cone} is 
\[ H(C) := \lim(C) \cap -\lim(C) = \{ x \in V \: C + x = C\}.\] 
We say that $C$ is {\it pointed} if $H(C) =\{0\}$, which is 
equivalent to $C$ not containing affine lines 
(\cite[Prop.~V.1.6]{Ne00}). 
\end{defn}

\begin{defn} A finite dimensional real Lie algebra $\g$ 
is called {\it admissible} if 
it contains an $\Inn(\g)$-invariant pointed generating 
closed convex subset $C$ (cf.~\cite[Def.~VII.3.2]{Ne00}).
\end{defn}

\begin{rem}
If $C \subeq \g$ is an invariant pointed closed convex subset, 
then $\g_C := \Spann C \trile \g$ 
is an ideal in which $C$ is also generating, so that $\g_C$ is an admissible 
Lie algebra. 
\end{rem}

\begin{defn} \mlabel{def:spind} (Spindler construction;  \cite{Sp88})
Let $\fl$ be a Lie algebra, 
$V$ an $\fl$-module, $\fz$ a vector space, and 
$\beta \: V \times V \to \fz$ an $\fl$-invariant skew-symmetric bilinear map.
Then $\fz \times V \times \fl$ is a Lie algebra with
respect to the bracket 
\begin{equation}
  \label{eq:VII.2.1}
 [(z,v,x), (z',v',x')] 
= (\beta(v,v'), x.v' - x'.v, [x,x']). 
\end{equation}
We write $\g(\fl, V, \fz, \beta)$ 
for the so-obtained Lie algebra. 
\end{defn}

The following theorem describes the structure of 
non-reductive admissible Lie algebras. 

\begin{thm} \mlabel{thm:spind} 
Any admissible Lie algebra $\g$ is of the form 
$\g(\fl,V,\fz,\beta)$, where 
\begin{itemize}
\item[\rm(a)] $\fz = \fz(\g)$, $\fu = \fz + V$ is the maximal nilpotent ideal. 
\item[\rm(b)] $\fl$ is reductive and quasihermitian, i.e., 
all simple ideals are compact or hermitian. 
\item[\rm(c)] $\fl$ contains a compactly embedded Cartan subalgebra $\ft_\fl$, 
and  $\ft := \fz + \ft_\fl$ is a compactly embedded Cartan subalgebra of $\g$.
\item[\rm(d)] $V = [\ft,\fu] = [\ft_\fl,V]$.
\item[\rm(e)] There exists an element $f \in \fz^*$ such that $(V, f \circ \beta)$ 
is a symplectic $\fl$-module of convex type, i.e., there exists an element 
$x \in \fl$, such that the Hamiltonian function 
\[ H_x^f : V \to \R,\quad H_x^f(v) := f(\beta(x.v,v)) \] 
is positive definite. 
\end{itemize}
Conversely, $\g(\fl,V,\fz,\beta)$ is admissible if {\rm(a)-(e)} are satisfied.
\end{thm}

\begin{prf} Properties (a)-(d) follows from 
\cite[p.~293]{Ne00}, combined with \cite[Thm.~VII.2.26]{Ne00}. 
Since $\ft$ is compactly embedded, (d) is equivalent to 
$\fz_V(\ft_\fl):= \{ v \in V \: [v,\ft_\fl] = \{0\}\} =  \{0\}.$ 
Further 
\[ \fz_{\fz(\fl)}(V) := \{ x \in \fz(\fl) \: [x,V] = \{0\}\} = 
 \fz(\fl) \cap \fz(\g) = \{0\}. \]
Therefore \cite[Thm.~VIII.2.7]{Ne00} implies that, 
$\g(\fl,V,\fz,\beta)$ is admissible if (a)-(e) are satisfied.
\end{prf}

\begin{cor}
  \mlabel{cor:brack-nondeg} {\rm(Non-degeneracy of $\beta$)} 
Let $\g = \g(\fl,V,\fz,\beta)$ be admissible and 
$0 \not=v \in V$. Then there exists an element $w \in V$ with 
$\beta(v,w) = [v,w] \not=0$.
\end{cor}

\begin{prf} In the context of Theorem~\ref{thm:spind}(e), 
we see that $w := x.v$ satisfies $\beta(x.v,v) \not=0$. 
  \end{prf}

\begin{lem} \mlabel{lem:abideal} 
If $\g = \g(\fl,V,\fz,\beta)$ is admissible, then 
the following assertions hold: 
\begin{itemize}
\item[\rm(a)] Every abelian ideal of $\g$ is central. 
\item[\rm(b)] Every ideal of $\g$ contained in $V$ is trivial.  
\end{itemize}
\end{lem}

\begin{prf} (a) Let $\fa \trile \g$ be an abelian ideal and 
$C \subeq \g$ be a pointed generating invariant closed convex subset. 
For $x \in C$ we then have 
$e^{\ad \fa} x = x + [\fa,x] \subeq C$, and, since $C$ is pointed,  
$[x,\fa] = \{0\}$. As $\g = \spann(C)$, it follows that 
$[\g,\fa] = \{0\}$. 

\nin (b) If $\fa \subeq V$ is an ideal of $\g$, 
then $[V,\fa] \subeq [V,V] \cap \fa \subeq \fz\cap \fa = \{0\}$ implies that 
$\fa$ is abelian. In view of (a), 
$\fa$ is central, so that $\fa \subeq V \cap \fz = \{0\}$. 
\end{prf}

\begin{defn} (a) 
Let $\ft \subeq \g$ be a compactly embedded Cartan 
subalgebra, $\g_{\C}$ the complexification of $\g$, 
$z= x + i y \mapsto z^* := - x + iy$
the corresponding involution, and $\ft_\C$ the 
corresponding Cartan subalgebra of $\g_{\C}$. 
For a linear functional $\alpha \in \ft_\C^*$, we define the
{\it root space} 
\[ \g_{\C}^\alpha := \{ x \in \g_{\C} : (\forall y \in \ft_\C)\ [y,x] = 
\alpha(y)x\} \] 
 and write  
\[ \Delta := \Delta(\g_\C, \ft_\C) 
:= \{ \alpha \in \ft_\C^*\setminus \{0\}: 
\g_\C^\alpha \not= \{0\}\} \] 
for the set of {\it roots of $\g$.} 

\nin (b)  A root $\alpha \in \Delta$ is called {\it semisimple} 
if $\alpha([z,z^*]) \not= 0$ holds for an element $z \in \g_\C^\alpha$. 
In this case, $[\g_\C^\alpha, \g_\C^{-\alpha}] = \C [z, z^*]$ contains a 
unique element $\alpha^\vee$
with $\alpha(\alpha^\vee) = 2$ which we call the 
{\it coroot of $\alpha$}. 
We write $\Delta_s$ for the set of semisimple roots and call the roots in 
$\Delta_r := \Delta \setminus \Delta_s$ the {\it solvable roots}.  

A semisimple root $\alpha$ is called {\it compact} 
if $\alpha^\vee \in \R^+ [z,z^*]$, i.e.\  if 
$\alpha([z, z^*]) > 0$. All other roots are called {\it non-compact}. 
We write $\Delta_k$, resp., $\Delta_p$ for the set of compact, resp., 
non-compact roots. We also set $\Delta_{p,s} := \Delta_p \cap
\Delta_s$.   

\nin (c) For each compact root $\alpha \in \Delta_k$, the linear mapping 
$ s_\alpha \: \ft \to \ft, x \mapsto x - \alpha(x) \alpha^\vee$ 
is a reflection in the hyperplane $\ker \alpha$. 
We write $\cW_\fk$ for the group generated by these
reflections. It is called the {\it Weyl group} 
of the pair $(\fk,\ft)$. 
According to \cite[Prop.~VII.2.10]{Ne00},
this group is finite. 
\end{defn}

\begin{defn} (a)  A subset $\Delta^+ \subeq \Delta$ is 
called a {\it positive system} 
if there exists an element $x_0 \in i\ft$ with 
\[  \Delta^+ = \{ \alpha \in \Delta : \alpha(x_0) > 0\} \] 
and $\alpha(x_0) \not= 0$ holds for all $\alpha \in \Delta$. 
A positive system $\Delta^+$ is said to be 
{\it adapted} 
if for $\alpha\in \Delta_k$ 
and $\beta \in \Delta_p^+$ we have $ \beta(x_0) > \alpha(x_0)$ for 
some $x_0$ defining $\Delta^+$. In this case, we call 
$\Delta_p^+ := \Delta^+\cap \Delta_p$ an {\it adapted system of
positive non-compact roots}. 

\nin (b) We associate to an adapted system $\Delta_p^+$ 
of positive non-compact roots the convex cones 
\[  C_{\rm min} := C_{\rm min}(\Delta_p^+) := 
\cone(\{ i[{z_\alpha}, z_\alpha^*] \: z_\alpha \in 
\g_\C^\alpha, \alpha\in \Delta_p^+ \}) \subeq \ft, \] 
and 
\[  C_{\rm max} := C_{\rm max}(\Delta_p^+) := \{ x \in  \ft \: (\forall \alpha \in 
\Delta_p^+)\, i \alpha(x) \geq 0\}. \] 
\end{defn} 

The structure theoretic concepts introduced above 
play a crucial role in the analysis of invariant convex 
subsets. In particular, \cite[Thm.~VII.3.8]{Ne00} asserts that 
the existence of a pointed generating $\Inn(\g)$-invariant closed convex cone 
$W \subeq \g$ implies the existence of a compactly embedded Cartan subalgebra 
(cf.\ Theorem~\ref{thm:spind}), 
and that, for every compactly embedded Cartan subalgebra 
$\ft$, there exists an adapted  positive system $\Delta_p^+$ with 
\[ C_{\rm min} \subeq W \cap \ft = p_\ft(W) \subeq C_{\rm max},\] where 
$p_\ft \: \g \to \ft$ denotes the projection with kernel 
$[\ft,\g]$. 
Moreover, $W$ is uniquely determined by $W \cap \ft$, 
the cone $C_{\rm min}$ is pointed. 
By \cite[Thm.~VIII.2.12]{Ne00} (cf.\ also \cite[Thm.~VIII.3.7]{Ne00}), 
for an adapted positive system 
$\Delta_p^+$ and an admissible Lie algebra, 
the pointedness of $C_{\rm min}$ 
implies that $C_{\rm min} \subeq C_{\rm max}$. 
Note that 
\[ C_{\rm min} = C_{\rm min,\fz} + 
\cone(\{ -i \alpha^\vee \: \alpha \in \Delta_{p,s}^+\}) \] 
is pointed if and only 
if $C_{\rm min,\fz}$ is pointed because 
$\{ \alpha^\vee \: \alpha \in \Delta_{p,s}^+\}$ is a finite 
subset contained in an open half space. 
If this condition is satisfied, then \cite[Prop.~VIII.3.7]{Ne00} 
shows that 
\[ W_{\rm max} := \{ x \in \g \: p_\ft({\cal O}_x) \subeq C_{\rm max} \} \] 
is a generating closed convex invariant cone with 
$W_{\rm max} \cap \ft = C_{\rm max}$, and 
\[  W_{\rm min} := \{ x \in \g \: p_\ft({\cal O}_x) \subeq C_{\rm
min}\} \] 
is a pointed, closed convex invariant cone with 
$W_{\rm min} \cap \ft = C_{\rm min}$. In general $W_{\rm min}$ is 
not generating. The most extreme situation occurs if $\g$ is a compact 
Lie algebra. 
Then $W_{\rm min} = \{0\}$ and $W_{\rm max} = \g$.

\section{Elements in pointed cones} 
\mlabel{sec:3}

In this section we study elements $x$ in an admissible Lie algebra 
$\g = \g(\fl,V,\fz,\beta)$ for which $\co(x)$ is pointed. 
Splitting $\fl$ into an ideal $\fl_0$ commuting with 
$V$ and an ideal $\fl_1$ acting effectively on $V$, 
reduces this problem to the two cases, where either $\g = \fl$ is reductive 
 (Subsection~\ref{subsec:3.2}) or 
 where the reductive subalgebra $\fl$ 
acts faithfully on $V$  (Subsection~\ref{subsec:3.3}).

\subsection{General observations} 
\mlabel{subsec:3.1}

Theorem~\ref{thm:spind} provides powerful structural 
information that is crucial to analyze the subsets 
$\g_{\rm c}$ and $\g_{\rm co}$ for non-reductive admissible Lie algebras. 
Throughout this section we write 
\[ \g = \fu \rtimes \fl 
= (\fz \oplus V) \rtimes \fl = \g(\fl,V,\fz,\beta) \quad \mbox{ with  } \quad 
\fz = \fz(\g) \quad \mbox{ and } \quad \beta(v,w) = [v,w].\]
The kernel 
$\fl_0 \trile \fl$ of the representation of $\fl$ on $V$ has a complementary 
ideal $\fl_1$, and $\g$ is a direct sum 
\begin{equation}
  \label{eq:g1g0}
\g = \g_1 \oplus \fl_0 
\quad \mbox{ with }  \quad 
\g_1 = \fz \oplus V \oplus \fl_1 = \g(\fl_1,V,\fz,\beta).
\end{equation}
Any $x \in \g$ decomposes accordingly as $x = x_1 + x_0$ with 
$x_1 \in \g_1$ and $x_0 \in \fl_0$, and 
\[ \co(x) = \co(x_1) \times \co(x_0)\] 
implies that 
\begin{equation}
  \label{eq:redux1}
\g_{\rm co} = \g_{1,{\rm co}} \times \fl_{0,{\rm co}}.
\end{equation}
This reduces the description of this set to the two cases, where 
$\g$ is reductive 
or the representation of $\fl$ on $V$ is faithful. 

\begin{lem}\mlabel{lem:v0v1}
For $x \in \fl$ we put 
\begin{equation}
  \label{eq:dualrel}
V_{x,0} := \{v \in V \: [x,v] = 0\} \quad \mbox{ and } \quad 
V_x := [x, V] \subeq V.
\end{equation}
Then the following assertions hold: 
\begin{itemize}
\item[\rm(a)] If $f \in \fz^*$ is such that $\omega := f \circ \beta$ 
is a symplectic form, then 
$V_x^{\bot_\omega} = V_{x,0}$ and $V_{x,0}^{\bot_\omega} = V_x$. 
\item[\rm(b)] 
$V_{x,0} = \{ v \in V \: [v,V_x] = \{0\}\} = V_x^{\bot_\beta}$ and 
$V_x = \{ v \in V \: [v,V_{x,0}] = \{0\}\} = V_{x,0}^{\bot_\beta}$.
\end{itemize}
  \end{lem}

  \begin{prf} 
(a) Since $\ad_V x := \ad(x)\res_V \in \sp(V,\omega)$, 
we have 
\[  V_x^{\bot_\omega} 
= \{  v \in V \: \omega(v, V_x) = \{0\}\}
= \{  v \in V \: \omega([x,v], V) = \{0\}\} = V_{x,0}.\] 
Now $V_{x,0}^{\bot_\omega} = V_x$ follows from $V_x \subeq V_{x,0}^{\bot_\omega}$ and 
$\dim V_{x,0}^{\bot_\omega} = \dim V - \dim V_{x,0} = \dim V_x$. 

\nin (b) For $v \in V_{x,0}$, we have 
$\beta(v,[x,V]) = - \beta([x,v],V)= \{0\}$, i.e., 
$\beta(V_{x,0}, V_x) =\{0\}$. Moreover, the non-degeneracy 
of $\beta$ (Corollary~\ref{cor:brack-nondeg}) 
shows that $\{0\} = \beta(v,[x,V]) = \beta([x,v],V)$ implies 
$[x,v] = 0$, i.e., $V_x^{\bot_\beta} = V_{x,0}$. 

To see that we also have $V_{x,0}^{\bot_\beta} \subeq  V_x$, 
we use (a) and the existence of an $f \in \fz^*$ 
for which $\omega := f \circ \beta$ 
is symplectic (Theorem~\ref{thm:spind}) to see that 
$V_{x,0}^{\bot_\beta} \subeq  V_{x,0}^{\bot_\omega}  = V_x.$ 
\end{prf}

For $y \in V$ and $x = x_\fz + x_V  + x_\fl \in \g$, we have 
\begin{align} \label{eq:Vconj}
e^{\ad y} x 
&= x + [y,x] + \frac{1}{2}[y,[y,x]] \notag\\ 
&= x + [y,x_V] + [y,x_\fl] + \frac{1}{2}(\ad y)^2 x_\fl \notag \\
&= \underbrace{\Big(x_\fz + [y,x_V] +  \frac{1}{2} [y, [y,x_\fl]]\Big)}_{\in \fz} 
+ \underbrace{(x_V -[x_\fl,y])}_{\in V} + \underbrace{x_\fl}_{\in \fl}.
\end{align}
This rather simple formula will be a key tool throughout this paper. 
We conclude in particular that 
$e^{\ad y}x \in \fz + \fl$ is equivalent to the vanishing of the 
$V$-component, i.e., to 
\[ x_V = [x_\fl,y].\] 
As $\Inn(\g) = e^{\ad V} \Inn_\g(\fl)$, 
\begin{equation}
  \label{eq:conjcrit}
 \cO_x  \cap (\fz + \fl)\not=\eset \quad \mbox{ if and only if } \quad 
x_V \in [x_\fl,V] = V_x.
\end{equation}
If this condition is satisfied, then \eqref{eq:Vconj}, applied with 
$x_V = 0$, shows that $e^{\ad V}x \cap (\fz + \fl)$ is a single element. 
For $x \in \fz + \fl$, we thus obtain 
\begin{equation}
  \label{eq:adorb}
 \cO_x  \cap (\fz + \fl) = \Inn_\g(\fl)x =: \cO_x^\fl.
\end{equation}

The following Reduction Theorem can be used to reduce many question 
concerning the sets $\co(x)$ and $C_x$ to elements in reductive Lie algebras. 
\begin{thm} \mlabel{thm:conj1} {\rm(Reduction Theorem)}
Let $\g\cong \g(\fl,V,\fz,\beta)$ be an admissible Lie algebra. 
If $\co(x)$ is pointed, then 
$\cO_x  \cap (\fz+ \fl)\not=\eset$. 
\end{thm}

\begin{prf} Write $x = x_\fz + x_V + x_\fl$. 
From \eqref{eq:Vconj} we derive in particular that 
$e^{\ad V_{x_\fl,0}} x = x + [V_{x_\fl,0}, x_V]$ 
is an affine subspace. If $\co(x)$ is pointed, this affine subspace 
is trivial, so that 
$[V_{x_\fl,0}, x_V]=\{0\}$. Now $x_V \in V_{x_\fl}$  
follows from Lemma~\ref{lem:v0v1}(b) and the assertion follows from \eqref{eq:conjcrit}. 
\end{prf}

\begin{cor} \mlabel{cor:1.7} 
Let $\g\cong \g(\fl,V,\fz,\beta)$ be an admissible Lie algebra. 
Then every $\ad$-nilpotent element 
$x\in \g$ with $\co(x)$ pointed is 
conjugate under inner automorphisms to an element of 
$\fz + \fs$ for $\fs= [\fl,\fl]$. 
Any $\ad$-nilpotent element of $\fz + \fl$ is contained in $\fz + \fs$. 
\end{cor}

\begin{prf} By Theorem~\ref{thm:conj1}, we may assume that 
$x = x_\fz + x_\fl \in \fz + \fl$, i.e., that $x_V = 0$. 
Then $\ad x = \ad x_\fl$ is nilpotent. 
Write $x_\fl = x_0 + x_\fs$, where $x_\fs \in \fs$ is nilpotent 
and $x_0 \in \fz(\fl)$. 
Since ${[x_0, x_\fs] = 0}$, $\ad(x_\fs)$ is nilpotent and 
$\ad(x_0)$ is semisimple because $\fz(\fl)\subeq \ft_\fl$ is compactly embedded, 
the decomposition 
$\ad(x_\fl) = \ad(x_0) + \ad(x_\fs)$ is the Jordan decomposition 
of $\ad(x_\fl)$. Therefore the nilpotency of this element 
implies $\ad(x_0) = 0$, hence that $x_0 = 0$ because 
$\fz(\fl) \cap \fz(\g) \subeq \fl \cap \fz = \{0\}$. We conclude that 
$x_\fl = x_\fs \in \fs$ and thus  $x \in \fz + \fs$. 
\end{prf}

\begin{cor} \mlabel{cor:2.6} Let $W \subeq \g$ be a pointed generating invariant 
cone. If $x = x_\fz + x_\fs \in W \cap (\fz(\g) + \fs)$ is 
$\ad$-nilpotent, then $x_\fz \in  W$ and $x_\fs  \in W$.
\end{cor}

\begin{prf} We may assume that the nilpotent element $x_\fs\in \fs$ is non-zero, 
otherwise the assertion is trivial. 
For any $h \in \fs$ with $[h,x_\fs] = 2 x_\fs$ 
(Jacobson--Morozov Theorem, \cite[Ch.~VIII, \S 11, Prop.~2]{Bo90}), 
we have 
\[  x_\fs = \lim_{t \to \infty} e^{-2t} e^{t \ad h} x \in W \quad \mbox{ and } \quad 
x_\fz = \lim_{t \to \infty} e^{-t \ad h} x \in W. \qedhere\] 
\end{prf}

\begin{lem}
If $x$ is a nilpotent element of the pointed generating invariant cone $W \subeq \g$ 
and $x_\fs \not=0$, then there exists a Lie subalgebra 
$\fm \subeq \g$, isomorphic to $\gl_2(\R)$, such that 
 $\fz(\fm) \subeq \fz(\g)$ and $\fm \cap W$ is pointed and generating. 
\end{lem}

\begin{prf} By Corollary~\ref{cor:1.7}, we may assume that
$x = x_\fz + x_\fs \in \fz(\g) + \fs$ holds for a Levi complement~$\fs$.
We first choose an $\fsl_2(\R)$-subalgebra 
$\fs_x \subeq \fs$ containing $x_\fs$ (cf.\ Proposition~\ref{prop:jac-mor}). If 
$x_\fz \not=0$, then 
\[ x = x_\fz + x_\fs \in \fm := \R x_\fz + \fs_x \cong \gl_2(\R).\] 
Further, Corollary~\ref{cor:2.6} implies that 
$\fm \cap W$ contains $x_\fz$ and $x_\fs$, hence is generating in $\fm$. 
\end{prf}

The following observation provides some information 
on the central part of $\lim(\co(x))$. 

\begin{lem} \mlabel{lem:centlimcon} 
For $x = x_\fz + x_V + x_\fl \in \g = \g(\fl,V,\fz,\beta)$, we consider the 
cone 
\[ C_{x,\fz} := C_{x_\fl,\fz} := \cone(\{[y,[y,x_\fl]] \: y \in V \}).\] 
Then 
\begin{equation}
  \label{eq:czlim}
 C_{x,\fz} \subeq \lim(\co(x)) \cap \fz.
\end{equation}
\end{lem}

\begin{prf}   For $t \to \infty$, formula \eqref{eq:Vconj} leads to 
\[ \lim_{t \to \infty} t^{-2}(e^{\ad t y}x) 
= \frac{1}{2}[y,[y,x]] 
= \frac{1}{2}[y,[y,x_\fl]] \in \lim(\co(x))  \cap \fz.\qedhere\] 
\end{prf}

\subsection{Reductive Lie algebras} 
\mlabel{subsec:3.2}

If $\g$ is a simple real Lie algebra 
and $\fk \subeq \g$ a maximal compactly embedded subalgebra, 
then the existence of a pointed generating invariant cone $W$ 
implies the existence of a non-zero element $z \in \fz(\fk)$, i.e., 
that $\g$ is {\it hermitian}. If this is the case, then 
\[ W_{\rm min} := C_z \] 
is a minimal invariant cone, the dual cone 
\[ W_{\rm max} := W_{\rm min}^\star 
:= \{ x \in \g \: (\forall y \in W_{\rm min}) \ \kappa(x,y) \geq 0\} \] 
with respect to the non-degenerate form $\kappa(x,y) = - \tr(\ad x \ad y)$ 
is a maximal invariant cone containing $W_{\rm min}$, 
and any other pointed generating invariant cone 
$W$ either satisfies 
\[ W_{\rm min} \subeq W \subeq W_{\rm max} \quad \mbox{ or }\quad 
W_{\rm min} \subeq -W \subeq W_{\rm max},\] 
depending on whether $z \in W$ or $-z \in W$ 
(see the Kostant--Vinberg Theorem in \cite[Thm.~III.4.7]{HHL89}, \cite{Vin80}).

\begin{lem} \mlabel{lem:cN} {\rm(\cite[Thm.~III.9]{HNO94}}
Let $\g$ be semisimple 
and $W \subeq \g$ be a pointed generating
 invariant cone. If $W' \supeq W$ is 
any pointed invariant cone containing $W$, then every 
nilpotent element $x \in W'$ is contained in $W$. 
\end{lem}

\begin{prop} \mlabel{prop:simpleherm} 
Let $\g$ be simple hermitian and $x \in \g$ with 
the Jordan decomposition $x = x_s + x_n$. Then 
the following are equivalent: 
\begin{itemize}
\item[\rm(a)] $C_x$ is pointed. 
\item[\rm(b)] $C_{x_s}$ and $C_{x_n}$ are pointed and, if $x_s \not=0$, 
then $x_n \in C_{x_s}$. 
\item[\rm(c)] $x \in W_{\rm max} \cup - W_{\rm max}$, where 
$W_{\rm max}$ is a maximal pointed invariant cone. 
\item[\rm(d)] $\co(x)$ is pointed. 
\end{itemize}
If $x_s \not=0$, then $C_x = C_{x_s}$. 
\end{prop}

\begin{prf} (a) $\Rarrow$ (b): 
If $C_x$ is pointed, then $x_s, x_n \in C_x$ follows from 
Corollary~\ref{cor:jac-mor}. 
If $x_s \not=0$, we see that $C_{x_s} \subeq C_x$ is a non-trivial 
invariant cone, and now  Lemma~\ref{lem:cN} shows that 
$x_n \in C_{x_s}$. 

\nin (b) $\Rarrow$ (a): If $x_s = 0$ and $C_{x_n}$ is pointed,
then $C_x = C_{x_n}$ is pointed. 
If $x_s \not=0$, we further assume that $x_n \in C_{x_s}$. 
Then $x = x_s + x_n \in C_{x_s}$ implies 
$C_x \subeq  C_{x_s}$, so that $C_x$ is pointed because $C_{x_s}$ is 
assumed to be pointed. 

\nin (a) $\Leftrightarrow$ (c) follows from the Kostant--Vinberg Theorem 
(\cite[Thm.~III.4.7]{HHL89}, \cite{Vin80}). 

\nin That (d) follows from (a) is clear. 
Suppose that $\co(x)$ is pointed. We may assume that 
$x \not=0$, and observe that this implies that $\co(x)$ has interior points. 
Let $\fk \subeq \g$ be maximal compactly embedded. Then 
$K := \Inn_\g(\fk) \subeq \Aut(\g)$ is compact and 
$\co(x)^K = \co(x) \cap \fz(\fk)$ contains an interior point~$z$. 
Let $W_{\rm max}$ be the maximal pointed generating invariant cone containing~$z$. 

We claim that $x \in W_{\rm max}$. 
The projection $p_{\fz(\fk)} \: \g \to \fz(\fk) = \fz_\g(\fk)$  
is the fixed point projection for the action of the compact 
group $K$. Therefore it preserves closed convex invariant 
subsets. This shows that $p_{\fz(\fk)} (\co(x)) \subeq \co(x)$. 
As $\lim(\co(-z)) = - C_z$, the subset 
$p_{\fz(\fk)} (\co(x)) \subeq \fz(\fk) = \R z$ must be contained 
in the half line $[0,\infty)z$. Therefore 
\[ x \in W := \{ y \in \g \: p_{\fz(\fk)}(\cO_y) \subeq [0,\infty)z \}.\] 
The right hand side is a closed convex invariant cone containing $z$. Therefore 
$x \in W \subeq W_{\rm  max}$ implies that $C_x$ is pointed.
\end{prf}

\begin{cor} \mlabel{cor:herm} If $\g$ is simple hermitian, then 
\[ \g_{\rm c} =\g_{\rm co} = W_{\rm max} \cup - W_{\rm max}.\] 
\end{cor}

\begin{prop} \mlabel{prop:3.10} 
{\rm($\g_{\rm co}$ for reductive Lie algebras)} 
Let $\g = \g_k + \g^1 + \cdots + \g^k$ be reductive, 
where $\g_k$ is the maximal compact 
ideal and the ideals $\g^j$ are simple non-compact. 
If $\g^j$ is hermitian, we write $W_{\rm max}^{\g_j}$ for a maximal proper 
invariant cone in $\g^j$ and otherwise we put $W_{\rm max}^{\g_j}:= \{0\}$. 
Then 
\[ \g_{\rm co} 
= \g_k \times \prod_{j = 1}^k (W_{\rm max}^{\g_j} \cup -W_{\rm max}^{\g_j}).\] 
\end{prop}

\begin{prf} This follows from $\g_{k,{\rm co}} =  \g_k$ and 
Corollary~\ref{cor:herm}, which entails 
$\g^j_{\rm co} =  W_{\rm max}^{\g_j} \cup -W_{\rm max}^{\g_j}$ 
for $j =1,\ldots, k$. 
\end{prf}

By the preceding proposition, the structure of 
$\g_{\rm co}$ is rather simple and adapted to the decomposition into 
simple ideals. The subset $\g_{\rm c}$ is slightly more complicated. 
From Corollary~\ref{cor:jac-mor}, we 
obtain the following characterization of elements 
contained in a given invariant cone $W$ in a reductive Lie algebra. 

\begin{prop} \mlabel{prop:1.1} {\rm(Reduction to nilpotent and semisimple elements)} 
Let $\g$ be a reductive Lie algebra and $x \in \g$  
be contained in the invariant cone $W \subeq \g$. 
Write $x = x_0 + x_s + x_n$  with $x_0 \in \fz(\g)$ and 
the Jordan decomposition $x_s + x_n$ of the component of $x$ in $[\g,\g]$. 
Then the following are equivalent: 
\begin{itemize}
\item[\rm(a)] $x \in W$. 
\item[\rm(b)] $x_0 + x_s \in W$ and $x_n \in W$. 
\end{itemize}
\end{prop}

\begin{prf} For $x \in W$,   Corollary~\ref{cor:jac-mor}, applied to the 
semisimple Lie algebra $[\g,\g]$, immediately 
  implies~(b). Conversely, (b) implies $x = x_0 + x_s + x_n \in W + W \subeq W$. 
\end{prf}

\begin{cor} \mlabel{cor:5.6} Let $\g$ be reductive and $x \in \g$. 
We write $x_n \in [\g,\g]$ for its nilpotent Jordan component 
and $x_s := x - x_n$ for its $\ad$-semisimple Jordan component. 
Then $C_x$ is pointed if and only if $\oline{C_{x_n} + C_{x_s}}$ is pointed. 
\end{cor}

\begin{prf} If $C_x$ is pointed, then 
Proposition~\ref{prop:1.1} implies that 
$x_s, x_n \in C_x$, so that $C_{x_n} + C_{x_s} \subeq C_x$ has 
pointed closure. The converse follows from 
$x = x_n + x_s \in \oline{C_{x_n} + C_{x_s}}$, which implies 
$C_x \subeq \oline{C_{x_n} + C_{x_s}}$. 
\end{prf}

\begin{lem}
  \mlabel{lem:forcharthm} 
Suppose that $\g$ is a reductive Lie algebra 
and $x \in \g$ is such that $C_x$ is pointed. 
Then the following are equivalent: 
\begin{itemize}
\item[\rm(a)] $0 \in \co(x)$. 
\item[\rm(b)] $x$ is a nilpotent element of $[\g,\g]$. 
\end{itemize}
\end{lem}

\begin{prf} (a) $\Rarrow$ (b): We write 
$\g = \fz(\g) \oplus \g^1 \oplus \cdots \g^n,$ 
where the $\g^j$ are simple ideals and accordingly 
\[ x = x_0 + x_1 + \cdots + x_n,\] 
so that $\co(x) = \{x_0\} \times \co(x_1) \times \cdots \times \co(x_n).$ 
Therefore $0 \in \co(x)$ implies that $x_0 = 0$ and that 
$0 \in \co(x_j)$ for every $j$. It follows that 
\[ C_x = \sum_{j= 1}^n C_{x_j},\] 
so that all cones $C_{x_j}\subeq \g^j$ are pointed. They are generating if $x_j\not=0$. 

It therefore suffices to show that, if $\g$ is simple hermitian 
and $0\not=x \in \g$ is such that $\co(x)$ is pointed with $0 \in \co(x)$, 
then $x$ is nilpotent. 
If $x$ is not nilpotent, then $x_s \not=0$.   
Proposition~\ref{prop:simpleherm}(b) shows that 
$x_n \in C_{x_s}$, and thus 
\[ x = x_s + x_n \in x_s + C_{x_s} \subeq C_{x_s} \]  leads to 
$\co(x) \subeq \co(x_s) + C_{x_s}$. 
We claim that $0 \not\in\co(x_s)$. 
The semisimplicity of $x_s$ implies that its orbit $\cO_{x_s}$ 
is closed (Theorem of Borel--Harish-Chandra, \cite[1.3.5.5]{Wa72}), 
and since $C_{x_s}$ is pointed, the orbit 
$\cO_{x_s}$ is admissible 
in the sense of \cite[Def.~VII.3.14]{Ne00}. 
Here we use that $\g \cong \g^*$ as $\Inn(\g)$-modules, 
so that adjoint orbits correspond to coadjoint orbits 
under a linear isomorphism. 
Next we use \cite[Prop.~VIII.1.25]{Ne00} to see that, 
if $p_\ft \: \g \to \ft$ is the projection onto a compactly 
embedded Cartan subalgebra~$\ft$, then 
\[ p_\ft(\co(x_s)) = \co(x_s) \cap \ft \subeq \conv(\cO_{x_s}).\] 
If $\co(x_s)$ contains $0$, it follows that $0\in\conv(\cO_{x_s})$, 
but as $\cO_{x_s}$ is contained in the convex set $C_{x_s} \setminus\{0\}$, 
this is a contradiction. 

\nin (b) $\Rarrow$ (a): If $x$ is nilpotent, then 
$0 \in \oline{\R_+ x} \subeq \oline{\cO_x}$ (Corollary~\ref{cor:jac-mor}) 
implies that $0 \in \co(x)$. 
\end{prf}

\begin{prop} {\rm($\g_{\rm c}$ for reductive admissible Lie algebras)} 
\mlabel{prop:3.14}
Let $\g = \g_k \oplus \g_p$ be a reductive Lie algebra, 
where $\g_k$ is the maximal compact ideal. 
We write elements $x \in \g$ accordingly as 
$x = x_k + x_p$. 
Then the following are equivalent: 
\begin{itemize}
\item[\rm(a)] $C_x$ is pointed. 
\item[\rm(b)] $C_{x_p}$ is pointed and, if $x_p$ is nilpotent and $x_k \not=0$, then 
$x_k \not\in [\g_k,\g_k]$. 
\item[\rm(c)] $\co(x)$ is pointed, and if $0\in \co(x)$, then $x_k = 0$ and 
$x_p$ is nilpotent. 
\end{itemize}
\end{prop}

\begin{prf} We have 
  \begin{equation}
    \label{eq:cosum}
 \co(x) = \co(x_k) + \co(x_p),
  \end{equation}
where $\co(x_k)$ is compact. Therefore 
$\co(x)$ is pointed if and only if $\co(x_p)$ is pointed, 
which by Proposition~\ref{prop:simpleherm}(d), applied to the simple ideals 
in $\g_p$, implies that $C_{x_p}$ is pointed. 

\nin (a) $\Rarrow$ (b):  If $C_x$ is pointed, then 
the argument above shows that $C_{x_p}$ is pointed as well. 
If $x_p$ is nilpotent, then $0 \in \co(x_p)$ by 
Lemma~\ref{lem:forcharthm}, so that \eqref{eq:cosum} 
implies that $C_{x_k} \subeq C_x$ is also pointed, 
and this further implies that, if $x_k \not=0$, then $x_k \not\in [\g_k,\g_k]$. 
Here we use that the relative interior of $C_{x_k}$ intersects $\fz(\g_k)$ 
because the projection $p_\fz \: \g_k \to \fz(\g_k)$ is the fixed point 
projection for the compact group $\Inn(\g_k)$. 

\nin (b) $\Rarrow$ (c): Suppose that $C_{x_p}$ is pointed. 
Then 
\[ \co(x) = \co(x_k) \times \co(x_p) \subeq \co(x_k) \times C_{x_p},\] 
and since $\co(x_k)$ is compact, $\lim(\co(x)) \subeq C_{x_p}$ 
is pointed. 
If $0 \in \co(x)$, then $0 \in \co(x_k)$ and $0\in\co(x_p)$. 
As $C_{x_p}$ is pointed, Lemma~\ref{lem:forcharthm} implies 
that $x_p$ is nilpotent. If $x_k \not=0$, then 
$x_k \not\in [\g_k,\g_k]$ by (b), contradicting 
$0 \in \co(x_k) \subeq x_k + [\g_k,\g_k].$ 

\nin (c) $\Rarrow$ (a):  Suppose that $\co(x)$ is pointed and that, if $0 \in \co(x)$, 
then $x = x_p$ is nilpotent. If $0 \not\in \co(x)$, then $C_x$ is pointed by 
Lemma~\ref{lem:conecrit}. If $0\in \co(x)$, then $x = x_p$ is nilpotent, 
so that $\co(x_p) = C_{x_p} = C_x$ is pointed. 
\end{prf}

\begin{prop} {\rm(Extension of invariant cones)} 
 Suppose that $\g$ is reductive and quasihermitian, 
i.e., a direct sum of a compact Lie algebra and hermitian simple ideals. 
If $C_x$ is pointed, then there 
exists a pointed generating invariant cone $W \subeq \g$ 
containing~$x$.
\end{prop}

\begin{prf} Let $\g(x) = C_x - C_x \trile \g$ be the ideal
generated by $x$. As $\g$ is reductive, 
$\g= \g(x) \oplus \g_1$, where $\g_1 \trile \g$ is a complementary 
ideal. If $\g_1$ itself contains a pointed generating invariant cone 
$W_1$, then $C_x + W_1$ is a pointed generating invariant cone in $\g$. 
As $\g_1$ also is quasihermitian, it contains no pointed generating 
invariant cone if and only if it is compact semisimple 
(Proposition~\ref{prop:3.14}). Then  we consider a product 
$B_x \times B_1 \subeq \g$, where $B_x \subeq C_x$ is a compact base 
of the pointed cone $C_x$, i.e., of the form $f^{-1}(1) \cap C_x$ for 
$f$ in the interior of the dual cone $C_x^\star$, 
and $B_1 \subeq \g_1$ is a compact invariant 
$0$-neighborhood. Then $W := \cone(B_x + B_1)$ is a pointed 
generating invariant cone in $\g$ containing $C_x = \cone(B_x)$. 
\end{prf}

We shall see below that the preceding proposition does not 
extend to non-reductive admissible Lie algebras 
(Example~\ref{ex:counterex}).

\subsection{The Characterization Theorem} 
\mlabel{subsec:3.3}

We now turn to the non-reductive admissible Lie algebras 
$\g = \g(\fl,V,\fz,\beta)$. 
With the reduction \eqref{eq:g1g0} in mind, we 
 may {\bf assume that $\fl$ acts faithfully on~$V$.}
We start with a crucial lemma. 

For a root $\alpha \in \Delta_r$ with root space 
\[ \g_\C^\alpha = 
V_\C^\alpha =  \{ w \in V_\C \: (\forall x \in \ft)\ [x,w] = \alpha(x)w\} \] 
(see \cite[Prop.~VII.2.5]{Ne00} for $V_\C^\alpha = \g_\C^\alpha$), we consider the closed convex cone 
\[  C_\alpha := \cone(\{  i[v_\alpha, v_\alpha^*] \: 
 v_\alpha \in V_\C^\alpha\})\subeq \fz,\] 
As $z \mapsto z^*$ exchanges $V_\C^\alpha$ and $V_\C^{-\alpha}$, we have 
\begin{equation}
  \label{eq:calpharel}
 C_{-\alpha} = C_\alpha.
\end{equation}

\begin{lem}
  \mlabel{lem:Wl}
For any pointed closed convex cone 
$C_\fz \subeq \fz$, 
\[ W_\fl := W_\fl(C_\fz) := \{x \in \fl \: (\forall v \in V)\ 
[v,[v,x]] \in C_\fz\} \]
is a pointed invariant closed convex cone in $\fl$. 
\end{lem}

\begin{prf} That $W_\fl$ 
is a closed convex invariant cone in $\fl$ follows from the 
$\fl$-invariance of the bracket $\beta \: V \times V \to \fz$. 
To see that $W_\fl$ is pointed, let $y \in W_\fl \cap - W_\fl$. 
Then $[v,[v,y]] = 0$ for every $v \in V$. 
As the map $V \times V \to \fz, (v,w) \mapsto [v,[w,y]]$ 
is symmetric by 
\[ [w,[v,y]] = [[w,v],y] - [[w,y],v] = [v,[w,y]],\] 
polarization 
shows that $[V,[V,y]] = \{0\}$. 
Next Corollary~\ref{cor:brack-nondeg} 
entails $[V,y] = \{0\}$, so that 
$y =0$  because  the representation 
of $\fl$ on $V$ is injective. This shows that $W_\fl$ is pointed. 
\end{prf}

\begin{lem} \mlabel{lem:czxt} For $x \in \ft$, we have 
$C_{x,\fz} = \oline{\sum_{i\alpha(x)>0} C_\alpha}.$
\end{lem}

\begin{prf}
We write $v \in V$ as 
$v= \sum_{\alpha \in \Delta_r^+} v_\alpha - v_\alpha^*$ with 
$v_\alpha \in V_\C^\alpha,$ 
where $\Delta_r^+ \subeq \Delta_r$ is any positive system. 
We then find for $x \in \ft$: 
\[  [v,x] = \sum_{\alpha \in \Delta_r^+} -\alpha(x) (v_\alpha + v_\alpha^*)\] 
and thus 
\begin{align*}
 [v,[v,x]] 
&= \sum_{\beta,\alpha \in \Delta_r^+} -\alpha(x) [v_\beta - v_\beta^*, v_\alpha + v_\alpha^*] 
= \sum_{\alpha \in \Delta_r^+} -\alpha(x) 
([v_\alpha, v_\alpha^*] - [v_\alpha^*, v_\alpha]) \\
&
= \sum_{\alpha \in \Delta_r^+} -2\alpha(x) [v_\alpha, v_\alpha^*] 
= \sum_{\alpha \in \Delta_r^+} 2i\alpha(x) i[v_\alpha, v_\alpha^*].
\end{align*}
This shows that 
$i\alpha(x) > 0$ implies that $C_\alpha \subeq C_{x,\fz}$. 
As $-i\alpha(x) C_{-\alpha} = i\alpha(x) C_\alpha$ by 
\eqref{eq:calpharel}, the cone 
$C_{x,\fz}$ is generated by the cones $C_\alpha$ for $i\alpha(x) > 0$, 
$\alpha \in \Delta_r$. 
\end{prf}

\begin{prop} \mlabel{prop:wl-gen}
Let $\Delta_r^+ \subeq \Delta_r$ be a positive system 
and 
\[ C_\fz \supeq C_{\rm min,\fz} 
:= \oline{\sum_{\alpha \in \Delta_r^+} C_\alpha} 
=  \cone(\{  i[v_\alpha, v_\alpha^*] \: 
\alpha \in \Delta_r^+, v_\alpha \in V_\C^\alpha\})\] 
be a pointed closed convex cone in $\fz$. 
Then 
\[ W_\fl := W_\fl(C_\fz) := \{x \in \fl \: (\forall v \in V)\ 
[v,[v,x]] \in C_\fz\} \]
is a pointed generating closed convex invariant cone in $\fl$ with 
\begin{equation}
  \label{eq:wlt}
W_\fl \cap \ft_\fl = (i\Delta_r^+)^\star.
\end{equation}
\end{prop}
By \cite[Thm.~VII.2.7]{Ne00}, any 
$x \in \ft_\fl$ for which some Hamiltonian function 
$H_x^f$, $f \in C_\fz^\star$,
 is positive definite is contained in the interior of $W_\fl$

\begin{prf} In view of Lemma~\ref{lem:Wl}, it remains to show that 
$W_\fl$ is generating. First we observe that 
\[ W_\fl = \{x \in \fl \: C_{x,\fz} \subeq C_\fz\}. \]
If $x \in \ft_\fl$ is such that 
$i\alpha(x) > 0$ for all $\alpha \in \Delta_r^+$, then 
$C_{x,\fz} = C_{\rm min,\fz}$ by Lemma~\ref{lem:czxt}. Therefore 
\[ W_\fl \cap \ft_\fl \supeq (i\Delta_r^+)^\star,  \] 
and this cone has inner points. This implies that $W_\fl$ is generating 
because $\Inn(\fl)(W_\fl\cap \ft_\fl)$ contains inner points. 
This in turn follows 
from the fact that the differential of the map 
\[ \Phi \: \fl \times \ft_\fl \to \fl,\quad 
\Phi(x)(y) := e^{\ad x}y \] 
is surjective if $y \in \ft_\fl$ is regular.

If $x \in \ft_\fl$ is such that $i\alpha(x) < 0$ for some 
$\alpha \in \Delta_r^+$, then 
$C_\alpha \subeq - C_{\rm min,\fz} \subeq - C_\fz$ 
and the pointedness of $C_\fz$ thus shows that 
$C_{x,\fz} \not\subeq C_\fz$. 
We conclude in particular that $x \not\in W_\fl$, so that 
$W_\fl \cap \ft \subeq (i\Delta_r^+)^\star$, and thus~\eqref{eq:wlt} follows. 
\end{prf}

\begin{rem} \mlabel{rem:conewcz} 
In the context of the preceding proposition, we 
assume, in addition, that $\Delta_p^+$ is adapted and 
consider the closed convex cone 
\[ W := \{ x \in \g \: \cO_x \subeq C_\fz + V + W_\fl\}.\] 
Then the pointedness of $C_{\rm min,\fz}$ implies the pointedness 
of $C_{\rm min}$ and hence that $C_{\rm min} \subeq C_{\rm max}$ 
(\cite[Thm.~VIII.2.12]{Ne00}). 
As 
\[ \Inn(\g) W_\fl  = e^{\ad V} W_\fl \subeq C_\fz + V + W_\fl\] 
follows from \eqref{eq:Vconj}, we have 
$C_\fz + W_\fl \subeq  W.$ 
For 
\[ \fz_W := C_\fz - C_\fz \supeq C_{\rm min,\fz} - C_{\rm min,\fz} = [V,V],\] 
the cone $W$ is clearly contained in the ideal 
$\g_W := \g(\fl,V,\fz_W,\beta)$. 
Further,
\[ C_\fz + (W_\fl \cap \ft_\fl) \subeq W \cap \ft \subeq C_\fz + (W_\fl \cap \ft_\fl) \] 
implies that 
\[ W \cap \ft = C_\fz + (W_\fl \cap \ft_\fl),\] 
and this cone is generating in $\ft_W := \ft \cap \g_W$, 
which in turn implies that $W$ is generating in $\g_W$. 
By construction, $H(W) \subeq V$ is trivial by 
Lemma~\ref{lem:abideal}(b), so that $W$ is pointed. 
\end{rem}

\begin{thm} \mlabel{thm:2.12} {\rm(Characterization Theorem)} 
Let $\g = \g(\fl,V,\fz,\beta)$ be admissible 
such that the representation 
of $\fl$ on $V$ is faithful. For $x = x_\fz + x_V + x_\fl \in \g$ we 
consider the $\fz$-valued Hamiltonian function 
\[ H_x^\fz \:  V \to \fz,\quad 
H_x^\fz(v) := p_\fz(e^{\ad v}x) = x_\fz + [v,x_V] +  \frac{1}{2} [v, [v,x_\fl]] \] 
and put 
\[  \coz(x) := \oline{\conv}(H_x^\fz(V)). \] 
Then 
    \begin{itemize}
    \item[\rm(a)] $\co(x)$ is pointed if and only if $\coz(x)$ is pointed, i.e., 
$\g_{\rm co} = \{ x \in \g \: \co_\fz(x)\ \mbox{\rm pointed} \}.$
    \item[\rm(b)] $C_x$ is pointed if and only if 
$\coz(x)$ is pointed and, if $x_\fl$ is nilpotent, 
then $\coz(x)$ is contained in a pointed cone. 
    \end{itemize}
\end{thm}

We shall see in Example~\ref{ex:counterex} below that 
the pointedness of $C_x$ does in general not 
imply the existence of a pointed generating 
invariant cone $W$ containing~$x$. 

\begin{prf}  (a) 
If $\co(x)$ is pointed, then we may assume that 
$x = x_\fz + x_\fl$ by Theorem~\ref{thm:conj1}. Note that 
$H_x^\fz$ only depends on the orbit $e^{\ad V}x$.  
Now $\co(x_\fl) = \co(x) -x_\fz$ is also pointed. 
That the cone $C_{x,\fz} = C_{x_\fl,\fz} = \coz(x_\fl)$ is pointed 
follows from 
$C_{x,\fz} \subeq \fz \cap \lim(\co(x))$ 
(Lemma~\ref{lem:centlimcon}). 

Suppose, conversely, that $\coz(x)$ is pointed. 
As 
\begin{equation}
  \label{eq:x.1}
 C_{x,\fz} \subeq \lim(\coz(x)),
\end{equation}
by Lemma~\ref{lem:centlimcon}, 
the cone $C_{x,\fz}$ is pointed. 
For $V_{x_\fl,0} := \{ v \in V \: [x_\fl,v] = 0\}$, 
the set $H_x^\fz(V_{x_\fl,0}) = x_\fz + [V_{x_\fl,0}, x_V]$ is an affine space. 
Our assumption implies that it is trivial, so that 
$[x_V,V_{x_\fl,0}] = \{0\}$ and thus $x_V \in [x_\fl, V]$ 
by Lemma~\ref{lem:v0v1}(b). 
Hence $\cO_x \cap (\fz \oplus \fs) \not=\eset$ 
follows from~\eqref{eq:conjcrit}. 
We may therefore assume that $x_V = 0$. 
Then $x = x_\fz + x_\fl$ and 
$\co(x) = x_\fz + \co(x_\fl)$. It therefore suffices to show that 
$\co(x_\fl)$ is pointed if the cone $C_{x_\fl,\fz}$ is pointed. 

By Lemma~\ref{lem:Wl}, 
\begin{equation}
  \label{eq:wl}
 W_\fl := \{ y \in \fl \: (\forall v\in V)\ [v,[v,y]] \in C_{x_\fl,\fz}\} 
\end{equation}
is a pointed invariant cone in $\fl$.  
As $x_\fl \in W_\fl$, we obtain 
\begin{equation}
  \label{eq:14}
 \co(x_\fl) \subeq C_{x_\fl,\fz} + V + W_\fl.
\end{equation}
Therefore the linear subspace $H(\co(x_\fl))$ is an ideal of $\g$ 
contained in $V$, hence trivial (Lemma~\ref{lem:abideal}(b)).
This shows that $\co(x_\fl)$ is pointed. 

\nin (b) Suppose that $C_x$ is pointed. Then 
$\co(x)$  is pointed and thus $\coz(x)$ is pointed by (a). 
Suppose that $x_\fl$ is nilpotent. We have to show that 
$\coz(x)$ is contained in a pointed cone. 
From  
\begin{equation}
  \label{eq:point2}
C_{x,\fz} \subeq \lim(\co(x)) \subeq C_x
\end{equation}
(Lemma~\ref{lem:centlimcon})  it follows that $C_{x,\fz}$ is pointed. 
As we have seen in (a), we may assume that $x_V = 0$, so that 
$x = x_\fz + x_\fl$. Further, the nilpotency of $x_\fl$ implies that 
it is contained 
in $\fs$ (Corollary~\ref{cor:1.7}).
Moreover, $0 \in \co(x_\fl)$ by Lemma~\ref{lem:forcharthm}, so that 
$C_x \supeq \co(x) = x_\fz + \co(x_\fl)\ni x_\fz.$ 
We conclude that 
\[\coz(x) 
= x_\fz + C_{x_\fl,\fz} 
= x_\fz + C_{x,\fz} \subeq C_x \] 
is contained in the pointed cone~$C_x\cap \fz$. 

Conversely, suppose that $\coz(x)$ is pointed 
and further that $x_V = 0$ by (a) and Theorem~\ref{thm:conj1}. 
\begin{itemize}
\item If $\coz(x)$ is contained in a pointed cone $D_x \subeq \fz$, then 
\[ \co(x)  \subeq D_x + V + \co_\fl(x_\fl).\] 
In the proof of (a) we have seen 
that $C_{x_\fl,\fz} \subeq \lim(\co_\fz(x))$ is pointed, and that $x_\fl$ is contained in the 
pointed invariant cone $W_\fl \subeq \fl$ from 
\eqref{eq:wl}. Therefore 
$C_x \subeq D_x  + V + W_\fl$ 
shows that $H(C_x) \trile \g$ is an ideal contained in $V$, 
hence trivial (Lemma~\ref{lem:abideal}(b)). 
\item If $\coz(x)$ is not contained in a pointed cone, then we assume that 
$x_\fl$ is not nilpotent. 
We show that $C_x$ is pointed by verifying that $0 \not\in \co(x)$ 
and applying Lemma~\ref{lem:conecrit}. 
As $\co(x) = x_\fz + \co(x_\fl),$ 
we have to show that $-x_\fz \not\in \co(x_\fl)$. 
We claim that $0 \not\in \co_\fl(x_\fl).$
In fact, if $x_\fl = y_s + y_n$ is the Jordan decomposition, 
where $y_n \in [\fl,\fl]$ is nilpotent and $y_s = x_\fl - y_n$, 
then $y_s, y_n \in W_\fl$ (Proposition~\ref{prop:1.1}). This shows that 
\[ \cO^L_{x_\fl} \subeq \cO^L_{y_s} + W_\fl
\quad \mbox{ implies } \quad 
 \co_\fl(x_\fl) \subeq \co_\fl(y_s) + W_\fl \subeq W_\fl.\] 
As $y_s \not=0$ by assumption 
and $W_\fl$ is pointed by (a), $0 \not\in \co_\fl(y_s)$ by 
Lemma~\ref{lem:forcharthm}, and this  implies that 
$0 \not\in \co_\fl(x_\fl)$. Finally we observe that 
\[ \co(x_\fl) \subeq \fz + V + \co_\fl(x_\fl),\]
and by the preceding argument, this convex set intersects $\fz$ trivially. 
Therefore $-x_\fz \not\in \co(x_\fl)$. 
\qedhere\end{itemize}
\end{prf}

\begin{rem} (a) From \eqref{eq:14} we derive in particular that 
the pointedness of $\co(x)$ implies the pointedness of 
$\co_\fl(x_\fl) \subeq \fl$, i.e., 
$p_\fl(\g_{\rm co}) \subeq \fl_{\rm co}.$ 

\nin (b) For $x = x_\fz + x_\fl \in \fz + \fl$, we have 
\[ \coz(x) = x_\fz + C_{x,\fz},\] 
which is pointed if and only if $C_{x,\fz}$ is pointed. 
If this is the case, then $\coz(x)$ is contained in a pointed cone if and only if 
\[ (x_\fz + C_{x,\fz}) \cap - C_{x,\fz} \subeq \{ 0\},\] 
which means that the intersection either is empty or $x_\fz = 0$ 
(cf.~Lemma~\ref{lem:pointconegen}). 
\end{rem}

\begin{rem} Let $f \in \fz^*$, considered as a linear functional on 
$\g = \g(\fl,V,\fz,\beta)$ via $f(z,v,x) := f(z)$. 
Then $f$ is fixed under the coadjoint action of $\Inn_\g(\fl)$, so that 
its coadjoint orbit is 
\[\cO_f = f \circ \Inn(\g) = f \circ e^{\ad V}. \] 
For $x = x_\fz + x_V + x_\fl \in \g$ we therefore have 
$\cO_f(x) = f(e^{\ad V}x) = f \circ H_x^\fz(V).$ 
This implies that 
\[ B(\cO_f) 
:= \{ x \in \g \: \inf \cO_f(x) > -\infty\} 
= \{ x \in \g \: \inf f(\co_\fz(x)) > -\infty\}.\] 
For elements $x = x_\fz + x_\fl \in \fz + \fl$, it follows that 
\[ x \in B(\cO_f) \quad \mbox{ if and only if } \quad 
f\in C_{x,\fz}^\star.\] 
\end{rem}

\begin{cor}\mlabel{cor:jacobialg} 
Let $\g :=\hsp_{2n}(\R)$ 
be the Jacobi--Lie algebra 
of inhomogeneous polynomials of degree $\leq 2$ on the symplectic 
vector space $(\R^{2n},\omega)$, endowed with the Poisson bracket. 
Then the following assertions hold: 
\begin{itemize}
\item[\rm(a)] $\co(x)$ is pointed if and only if 
the corresponding function $H_x$ on $\R^{2n}$ is semibounded, 
i.e., bounded from below or above. 
\item[\rm(b)] $C_x$ is pointed if and only if $x$ or $-x$ satisfies 
  \begin{itemize}
  \item[\rm(i)] $H_x$ is bounded from below, and 
  \item[\rm(ii)] if $x_\fl\not=0$ is nilpotent, then $H_x \geq 0$. 
  \end{itemize}
\end{itemize}
\end{cor}

\begin{prf} We shall obtain this is a special case of Theorem~\ref{thm:2.12}. 
The Jacobi algebra is admissible 
of the form $\g(\fl,V,\fz,\beta)$ with 
\[ \fl = \sp_{2n}(\R), \quad V = \R^{2n}, \quad 
\fz = \R \quad \mbox{ and } \quad \beta = \omega, \quad 
\omega((\bp_1, \bq_1),(\bp_2, \bq_2)) =  \bp_1 \bq_2- \bp_2 \bq_1.\] 
Further, 
\[ H^\fz_x(v) 
= x_\fz + \omega(v,x_V) + \frac{1}{2} \omega(x_\fl.v,v) = H_x(v) \] 
is the Hamiltonian function on $(V,\omega)$, corresponding to $x \in \g$. 

\nin (a) We conclude that $\co_\fz(x)\subeq \fz = \R$ is the closed convex hull 
of the range of~$H_x^\fz$. This is a pointed convex set if and only if 
$H_x$ is semibounded. 

\nin (b) means that $C_x$ is pointed if and only if 
$H_x$ is semibounded  and, if $x_\fl$ is nilpotent, then $H_x \geq 0$ 
(if $H_x$ is bounded from below) 
or $H_x \leq 0$ (if $H_x$ is bounded from above). 
So (b) follows from Theorem~\ref{thm:2.12}(b). 
\end{prf}

\begin{ex} For $V = \R^2$ and $\fs = \sp_2(\R) = \fsl_2(\R)$, the 
Hamiltonian function associated to 
\[ x = \pmat{a & b \\ c & -a} 
\quad \mbox{ is given by } \quad  H_x(q,p) = bp^2 - c q^2 + 2a pq.\] 
We have $H_x \geq 0$ if and only if $b,-c \geq 0$ and $a^2 \leq -bc$. 
Then $x$ is either elliptic or nilpotent, where the latter is equivalent to 
$a^2 =-bc$. 

Accordingly, an element $x = x_\fz + x_V + x_\fs \in \hsp_2(\R)$ 
with $H_{x_\fs} \geq 0$ generates a pointed cone 
if either $x_\fs$ is elliptic, i.e., 
positive definite, or if $x_\fs$ is nilpotent and $H_x \geq 0$.  
\end{ex}

\begin{ex} \mlabel{ex:counterex} 
(Elements in $\g_c$, not contained in a pointed generating cone) 
We consider the Lie algebra $\g = \g(\fl,V,\fz,\beta)$, where 
\[\fl = \R^2, \quad V = V_1 \oplus V_2 \oplus V_{1,2} \quad \mbox{ with } 
\quad V_1 = V_2 = \C, V_{1,2} = \C^2, \quad  \fz = \R^2, \] 
and the action of $\fl$ on $V$ is given by 
\[ (x_1,x_2).(z_1, z_2, z_3, z_4) 
= (i x_1 z_1, i x_2 z_2, i(x_1+x_2)z_3, i(x_1 + x_2) z_4).\] 
With $\eps_j(x_1, x_2) = i x_j$, this means that 
\[ \Delta = \Delta_r = \{ \pm \eps_1, \pm \eps_2, \pm (\eps_1 + \eps_2)\}. \]
We define $\beta \: V \times V \to \fz$ by 
\[ \beta(\bz, \bw) 
:= \big(\Im(\oline{z_1}w_1) + \Im(\oline{z_3}w_3), 
\Im(\oline{z_2}w_2) + \Im(\oline{z_4}w_4)\big) \in \R^2 = \fz.\] 
With $V^{[\alpha]} := V \cap (V_\C^\alpha + V_\C^{-\alpha}),$ 
we then have 
\[ V^{[\eps_1]} = V_1, \quad V^{[\eps_2]} = V_2 \quad \mbox{ and } \quad 
 V^{[\eps_1 + \eps_2]} = V_{1,2}.\] 
For $y:= (-1,0) \in \fl$, we obtain with $i\eps_1(y) = 1$: 
\[  C_{\eps_1} 
= \cone \{ [v,[v,y]] \: v \in V_1\}
= \cone \{ \beta(-iv,v) \: v \in  V_1\} = [0,\infty) \be_1.\]
We likewise obtain $C_{\eps_2} = [0,\infty) \be_2$, 
and 
\[  C_{\eps_1 + \eps_2} 
= \cone \{ \beta(-iv,v) \: v \in  V_{1,2}\} 
= [0,\infty) \be_1 + [0,\infty) \be_2.\]
Therefore $\Delta_r^+ := \{ \eps_1, \eps_2, \eps_1 + \eps_2 \}$ 
is an adapted positive system for which 
\[ C_{\rm min} =  [0,\infty) \be_1 + [0,\infty) \be_2 
\subeq \fz \subeq C_{\rm max} \]
is pointed. The invariant cone 
\[ W_{\rm min} := \{ x \in \g \: p_\ft(\cO_x) \subeq C_{\rm min} \} \] 
satisfies $W_{\rm min} \cap  \ft \subeq C_{\rm min}$, so that 
$H(W_{\rm min}) \trile \g$ is an ideal contained in $V$. 
As $\beta \: V \times V \to \fz$ is non-degenerate, this ideal 
is trivial and therefore $W_{\rm min}$ is pointed. 
Now \cite[Lemma~VIII.3.22]{Ne00} implies that $\g$ is admissible. 

Consider the element $x := (1,-1) \in \fl$. 
As $(\eps_1 + \eps_2)(x) = 0$, 
$i\eps_1(x) = -1$ and 
$i\eps_2(x) = 1$, the cone 
\[ C_{x,\fz} = C_{\eps_2} - C_{\eps_1} = \cone(\be_2, -\be_1) \] 
is pointed, so that $C_x$ is pointed by the 
Characterization Theorem~\ref{thm:2.12}. 

We claim that there exists no pointed generating invariant cone 
$W \subeq \g$ containing $x$. Suppose that $W$ is such a cone. 
Then there exists an adapted positive system $\Delta_r^+$ 
with $C_{\rm  min,\fz} \subeq W$ (\cite[Thm.~VII.3.8]{Ne00}). 
As $C_{x,\fz} \subeq C_x \subeq W$ (Lemma~\ref{lem:centlimcon}), we must have 
$\eps_2, -\eps_1 \in \Delta_r^+.$ 
If $\eps_1+ \eps_2$ is positive, then 
\[ C_{\eps_1} \subeq C_{\eps_1 + \eps_2} \subeq C_{\rm min,\fz} 
\quad \mbox{ and } \quad -C_{\eps_1} = C_{-\eps_1} \subeq C_{\rm min,\fz} \] 
contradict the pointedness of $C_{\rm min,\fz}$. If 
$\eps_1+ \eps_2$ is negative, then 
\[ -C_{\eps_2} \subeq C_{-\eps_1 - \eps_2} \subeq C_{\rm min,\fz} 
\quad \mbox{ and } \quad C_{\eps_2}  \subeq C_{\rm min,\fz} \] 
contradict the pointedness of $C_{\rm min,\fz}$. 
Hence there exists no pointed generating invariant cone $W$ containing~$x$. 
\end{ex}

In the preceding example it was important that $\dim \fz > 1$. 
We have the following positive result for the 
Jacobi--Lie algebra, where $\fz = \R$. 

\begin{prop} If $\g = \hsp_{2n}(\R) = \g(\sp_{2n}(\R), \R^{2n},\R,\omega)$ 
and $x \in \g$ is such that $C_x$ is pointed, then $x$ is contained 
in a pointed generating invariant cone $W \subeq \g$. 
\end{prop}

\begin{prf} In view of Corollary~\ref{cor:jacobialg}, we may assume that 
the Hamiltonian function $H_x$ is bounded from below. 
We write $x = x_\fz + x_V + x_\fl$ with $x_\fz \in \fz = \R$, 
$x_V \in V = \R^{2n}$ and $x_\fl \in \sp_{2n}(\R)$. 
If $x_\fl$ is nilpotent, then even $H_x \geq 0$, so that 
\[ x \in W := \{ y \in \g \: H_y \geq 0\},\] 
and $W$ is a pointed generating invariant cone in~$\g$.

We may therefore assume that $x_\fl$ is not nilpotent.
By the Reduction Theorem~\ref{thm:conj1}, we may further assume that 
$x_V = 0$. If $H_x \geq 0$, then $x \in W$; so we assume that 
$x_\fz = \min H_x(V) = H_x(0) < 0$. 
We now show that 
\[ W \cap -C_x = \{0\}.\] 
As $x \in \fz + V + W_\fl$ for $W_\fl := W \cap \fl$ 
(the cone of non-negative quadratic forms), the invariance 
of the set on the right implies $C_x \subeq \fz + V + W_\fl$. 
We conclude that $W \cap - C_x \subeq \fz + V$ 
is a pointed invariant cone. 
As $e^{\ad V}x = x + [V,x]$ for $x \in \fz + V$, 
it follows that $W \cap - C_x \subeq \fz$. We thus obtain 
\[ W \cap - C_x = (W\cap \fz) \cap (-C_x \cap \fz).\] 
If $C_{x,\fz} = \{0\}$, then $x_\fl = 0$ and $x = x_\fz \in - W$. 
So we may also assume that $C_{x,\fz} \not=\{0\}$. 
As $H_x$ is bounded from below, $C_x \supeq C_{x,\fz} = [0,\infty)$, so that 
we must have $C_x \cap \fz \subeq [0,\infty)$, whence 
\[ W \cap - C_x \subeq [0,\infty) \cap (-\infty,0] = \{0\}.\]  
Now \cite[Prop.~V.1.7]{Ne00} implies that the invariant cone 
$W + C_x \subeq \g$ is closed and pointed. It is generating because 
$W$ is generating. 
\end{prf}

\section{Affine pairs} 
\mlabel{sec:4}

In this section we turn to affine pairs related to invariant 
cones. We refer to the introduction for the motivation to study such 
pairs. 

\begin{defn} {\rm(Affine pair)} \mlabel{def:4.1}
Let $\g$ be a finite dimensional real Lie algebra and $W \subeq \g$ a 
pointed invariant cone. We call $(x,h) \in \g \times \g$ 
an {\it affine pair} for the cone $W$ if 
\begin{equation}
  \label{eq:affpa}
 x \in W \quad \mbox{ and } \quad [h,x] = x.
\end{equation}
\end{defn}

For an affine pair, the subalgebra 
$\R h + \R x$ is isomorphic to the non-abelian 
$2$-dimensional Lie algebra $\aff(\R)$; hence the name. 
As this Lie algebra is solvable, $\ad x$ is nilpotent 
(\cite[Prop.~5.4.14]{HN12}).

\subsection{Invariance of $W$ under 
one-parameter groups of outer automorphisms} 

On $\g = \g(\fl,V,\fz,\beta)$ we consider the canonical derivation 
$D_{\rm can}$, defined by 
\begin{equation}
  \label{eq:DV1}
 D_{\rm can}(z,v,x) := \Big(z,\frac{1}{2}v,0\Big). 
\end{equation}
The derivation $2 D_{\rm can}$ 
corresponds to the $\Z$-grading of $\g$, defined by 
\[ \g_0 = \fl, \quad \g_1 = V \quad \mbox{ and } \quad \g_2 = \fz.\] 
In the Existence Theorem~\ref{thm:4.2} below, 
the one-parameter group $e^{\R D}$ with $D \in D_{\rm can} + \ad \g$ 
leaves $W$ invariant if and only if 
$e^{\R D_{\rm can}}$ does. 
The following proposition characterizes the cones $W$ for which 
this is the case.

\begin{prop} \mlabel{prop:invcrit}
For a pointed generating 
invariant cone $W \subeq \g =\g(\fl,V,\fz,\beta)$, 
the following are equivalent: 
\begin{itemize}
\item[\rm(a)] $e^{\R D_{\rm can}} W = W$. 
\item[\rm(b)] $p_\fz(W) \subeq W$ and $p_\fl(W) \subeq W$. 
\item[\rm(c)] $e^{\R D_{\rm can}} (W\cap\ft) = W \cap \ft$. 
\end{itemize}
\end{prop}

If $W$ satisfies these conditions, then 
\[ W \cap (\fz + \fl) = W_\fz + W_\fl \quad \mbox{ for } \quad 
W_\fz := W \cap \fz\quad \mbox{ and } \quad 
W_\fl := W \cap \fl,\] 
and the Reduction Theorem~\ref{thm:conj1} 
implies that 
\[ W = e^{\ad V}.(W_\fz + W_\fl) = \oline{\Inn(\g).(W_\fz + W_\fl)}\] 
(cf.~\cite[Thm.~VII.3.29]{Ne00}), 
showing that $W$ is uniquely determined by the two cones 
$W_\fz$ and $W_\fl$.

\begin{prf} (a) $\Rarrow$ (b): 
For $x = x_\fz + x_V + x_\fl \in W$ we have for $t \to \infty$ 
\[ e^{-t}(e^{t D_{\rm can}}x) 
= x_\fz + e^{-t/2} x_V + e^{-t} x_\fl \to x_\fz
\quad \mbox{ and } \quad  e^{-t D_{\rm can}}x 
= e^{-t} x_\fz + e^{-t/2} x_V + x_\fl \to x_\fl.\]  

\nin (b) $\Rarrow$ (a): Let $x = x_\fz + x_V + x_\fl \in W$. Then 
Theorem~\ref{thm:conj1} implies the existence of 
$\phi \in \Inn(\g)$ with $y := \phi(x) \in \fz + \fl$. 
Then $y_\fz \in W$ and $y_\fl \in W$ by (a). Therefore 
$e^{tD_{\rm can}} y = e^{t} y_\fz + y_\fl \in W$  for $t \in \R.$ 
Now 
\[ e^{tD_{\rm can}} x 
= e^{tD_{\rm can}} \phi^{-1}(y) 
= e^{tD_{\rm can}} \phi^{-1} e^{-tD_{\rm can}} (e^{tD_{\rm can}}y) 
\in \Inn(\g) W = W.\] 

\nin (a) $\Rarrow$ (c) follows from $D_{\rm can}(\ft) \subeq \ft$. 

\nin (c) $\Rarrow$ (a): As $W$ is pointed and generating, 
\cite[Thm.~VII.3.29]{Ne00} yields 
$W = \oline{\Inn(\g)(W \cap \ft)},$ 
and the invariance of $W$ under $e^{\R D_{\rm can}}$ follows from (c). 
\end{prf}

\begin{ex} \mlabel{ex:hsp}
Not every pointed generating invariant cone satisfies 
$p_\fz(W) \not\subeq W_\fz$: 
For the Jacobi algebra 
\[ \g = \hsp_2(\R) = \g(\sp_2(\R), \R^2, \R, \omega) 
\quad \mbox{ and } \quad 
\omega(\bx,\by) = x_1 y_2 - x_2 y_1, \] 
we have a $2$-dimensional compactly embedded Cartan subalgebra 
\[ \ft = \fz \oplus \ft_\fs \cong \R^2 \quad \mbox{ and } \quad \cW_\fk = \{\id_\ft\}.\]
Up so sign, there is a unique positive system 
$\Delta^+$ (which is adapted). Then 
\[ C_{\rm min} = C_{\rm min,\fz} \oplus C_{\rm min,\fs} \] 
is a quarter plane and 
\[ C_{\rm max} = \R \oplus C_{\rm max,\fs} = \R \oplus C_{\rm min,\fs} \] 
is a half plane. 
Any pointed generating closed convex cone $W_\ft \subeq \ft$ with 
\[ C_{\rm min} \subeq W_\ft \subeq C_{\rm max} \] 
is of the form $W_\ft = W \cap \ft$ for a pointed generating 
invariant cone $W \subeq \g$ because the Weyl group $\cW_\fk$ 
is trivial (\cite[Thm.~VIII.3.21]{Ne00}). 
In particular, we may have 
$p_\fz(W_\ft) = \fz(\g) \not\subeq W_\ft$. Therefore 
we do not always have 
$p_\fz(W) \subeq W.$ 
\end{ex}

\begin{rem} Although the conditions in Proposition~\ref{prop:invcrit} 
are not always satisfied, this is the case for many naturally constructed 
cones. 

Let $\g = \g(\fl,V,\fz,\beta)$ be an admissible Lie algebra and 
$D \in \der(\g)$. We assume that the representation of $\fl$ on $V$ is faithful. 
The cone $W$, constructed from a pointed cone $C_\fz \subeq \fz$ 
in Remark~\ref{rem:conewcz} is generated by 
$C_\fz + W_\fl$ and satisfies 
$p_\fz(W) \subeq C_\fz \subeq W$ and $p_\fl(W) \subeq W_\fl\subeq W.$ 
Therefore $W$~is invariant under $e^{\R D_{\rm can}}$. 

More generally, any derivation $D \in \der(\g)$ with 
$\fl \subeq \ker D$ and $e^{\R D} C_\fz = C_\fz$ satisfies 
$e^{\R D} W = W$ because 
\[ W = e^{\ad V}(W \cap (\fz + \fl)) 
= \Inn(\g)(W \cap (\fz + \fl))  \] 
follows from the Reduction Theorem~\ref{thm:conj1}. 
\end{rem}

\subsection{Extending nilpotent elements to affine pairs} 

Let $W \subeq \g$ be a pointed generating invariant cone. 
In this section we consider a nilpotent element $x \in W$ 
and ask for the existence of a derivation $D \in \der(\g)$ with 
\[ Dx = x \quad \mbox{ and } \quad e^{\R D} W = W.\] 
Note that the latter condition implies that $W$ is an invariant cone 
in the extended Lie algebra $\g_D := \g \rtimes \R D$ 
and $(x,D)$ is an affine pair for~$W$. 

This problem is trivial for semisimple Lie algebras:
\begin{rem} \mlabel{rem:jac-mor} 

\nin (a) If $\g$ is semisimple, $W \subeq \g$ is an invariant cone, 
 and $x \in \g$ is nilpotent, then the 
Jacobson--Morozov Theorem (\cite[Ch.~VIII, \S 11, Prop.~2]{Bo90}) 
implies the existence of elements $h,y \in \g$ with 
\[ [h,x] = x, \quad [h,y] = -y \quad \mbox{ and } \quad [x,y]= h. \]
Then $D := \ad h$ is a derivation with $Dx = x$ and $e^{\R D}W = W$. 

\nin (b) If $\g = \fz(\g) \oplus [\g,\g]$ is reductive 
and $x = x_\fz + x_\fs$ with $0 \not= x_\fz \in \fz(\g)$ and 
$x_\fs \in[\g,\g]$, then it cannot be 
reproduced with inner derivations. For any derivation 
$D$ on $\g$ there exists an endomorphism $D_\z$ of $\fz(\g)$ and 
an element $h \in [\g,\g]$ with 
\[  D(z + x) = D_\fz(z) + [h,x] \quad \mbox{ for } \quad z \in \z(\g), 
x \in \g,\] 
so that $Dx = x$ is equivalent to $D_\fz x_\fz = x_\fz$ and 
$[h,x_\fs] = x_\fs$. If $W\ni x$ is a pointed generating invariant cone, 
then the nilpotency of $x_\fs$ implies that 
${\R_+ x_\fs \subeq  \cO_{x_\fs}}$ (Corollary~\ref{cor:jac-mor}), 
so that $x_\fz, x_\fs \in W$. 
Putting $D_\fz := \id_{\fz(\g)}$, we then have $Dx = x$ and 
at least $e^{\R D}(W_\fz + W_\fs) \subeq W$ for 
$W_\fz := W \cap \fz(\g)$ and $W_\fs := W \cap [\g,\g]$. 

As the maximal cone $W_{\rm max} \subeq \g$ contains $\fz(\g)$, 
it is invariant under $e^{\R D}$ for any $D \in \der(\g)$. 
\end{rem}

The following lemma provides crucial information that we shall 
need below to explore the existence of Euler derivations on 
$\g(\fl,V,\fz,\beta)$, i.e., a diagonalizable derivation 
with $\Spec(D) \subeq \{0,\pm 1\}$. 

\begin{lem}
  \label{lem:conv-module-euler-elmnt}
  Let \(\fl\) be a reductive quasihermitian Lie algebra and 
\(\fs \subset \fl\) be a semisimple subalgebra invariant 
under a compactly embedded Cartan subalgebra $\ft_\fl$  
such that every simple ideal in \(\fs\) is hermitian of tube type.
Let \((V,\omega)\) be a symplectic \(\fl\)-module of convex type 
and consider the subspaces
  \[V_{\rm eff,\fs} := \spann (\fs.V) \quad \text{and} \quad V_{\fix,\fs} := \{v \in V : \fs.v = \{0\}\}.\]
    Then the following assertions hold: 
    \begin{enumerate}
      \item[\rm(a)] \(V = V_{\rm eff,\fs} \oplus V_{\fix,\fs}\) is an \(\omega\)-orthogonal direct sum of \(\fs\)-submodules and the submodule \(V_{\rm eff,\fs}\) is a symplectic module of convex type for $\fs$.
      \item[\rm(b)] For every Euler element \(h \in \fs\) 
for which \(\fs_{\pm 1}(h)\) 
generate \(\fs\),  the operator $2 \ad h$ defines an 
antisymplectic involution on~\(V_{\rm eff,\fs}\).
    \end{enumerate}
\end{lem}

\begin{proof}
  (a) Since \((V,\omega)\) is a symplectic $\fl$-module of 
convex type with respect to the action of \(\fl\), the invariant cone
  \[W_{V,\fl} = \{x \in \fl : (\forall v \in V)\, \omega(x.v,v) \geq 0\}\]
  is generating, hence intersects the Cartan subalgebra $\ft_\fl$. 
Therefore $(V,\omega)$ also is a symplectic module of convex type 
for the reductive subalgebra $\fl' := \ft_\fl + \fs$. 
The edge $\fl_1$ of the generating invariant cone 
\[W_{V,\fl'} = \{x \in \fl' : (\forall v \in V)\, \omega(x.v,v) \geq 0\}\]
is the kernel of the representation of $\fl'$ on $V$, 
hence an ideal. We write 
\begin{equation}
  \label{eq:l'}
 \fl' = \fl_1 \oplus \fl_2 
\end{equation}
with a complementary ideal $\fl_2$. Then 
\[W_{V,\fl'} = \fl_1 \oplus W_{V,\fl_2} \quad \mbox{ with } 
W_{V,\fl_2} \ \mbox{ pointed and generating}.\] 
As $\fs = [\fl', \fl']$ adapts to the decomposition~\eqref{eq:l'}, 
\[ \fs = \fs_1 \oplus \fs_2 \quad \mbox{ with }\quad \fs_j := \fs \cap \fl_j.\]
Now 
\[W_{V,\fs} = \{x \in \fs : (\forall v \in V)\, \omega(x.v,v) \geq 0\}
= \fs_1 \oplus W_{V,\fs_2},\] 
where the pointed cone $W_{V,\fs_2} = W_{V,\fl_2} \cap \fs_2$ is also 
generating because it contains $W_{\rm min,\fs_2}$. 
Therefore $(V,\omega)$ is a symplectic $\fs_2$-module 
of convex type. 
As $V$ is a semisimple $\fs$-module, 
\[ V 
= V_{\rm eff,\fs} \oplus V_{\rm fix,\fs}
= V_{\rm eff,\fs_2} \oplus V_{\rm fix,\fs_2},\] 
where $V_{\rm eff,\fs}$ is a symplectic $\fs_2$-module 
of convex type because 
$W_{V,\fs_2} = W_{V_{\rm eff},\fs}$ is pointed and generating 
(\cite[Prop.\ II.5]{Ne94}). 

  (b) Let \(h \in \fs\) be an Euler element for which \(\fs_{\pm 1}(h)\) 
generate \(\fs\). 
  We decompose \(\fs\) into a direct sum 
\(\bigoplus_{j=1}^n \fs_j\) of simple ideals, which are 
hermitian of tube type because they possess Euler elements 
(\cite[Prop.~3.11(b)]{MN21}). 
  By \cite[Thm.\ 2.14]{Oeh20b}, we can decompose \(V_{\rm eff,\fs}\) into a direct sum \(\bigoplus_{j=1}^n V_j\) of \(\fs\)-submodules such that \(\fs_k\) acts trivially on \(V_j\) for $k \neq j$. 
  In particular, each \((V_j,\omega)\) is a symplectic $\fs_j$-module of convex type.
  We decompose \(h\) as \(h = \sum_{j=1}^n h_j\), with \(h_j \in \fs_j\).
Then each \(h_j\) is an Euler element in \(\fs_j\).
  Hence \cite[Lem.\ 3.4]{Oeh20b} implies that, for each $j$,  the operator 
$2 \ad h_j$ defines an antisymplectic involution on \(V_j\), 
and thus \(2 \ad h\) defines an antisymplectic involution on~\(V_{\rm eff,\fs}\).
\end{proof}

\begin{thm} {\rm(Existence Theorem)}  \mlabel{thm:4.2}
Let $\g = \g(\fl,V,\fz,\beta)$ be an admissible Lie algebra 
and $x = x_\fz + x_\fl\in \fz + \fl$ be  an 
$\ad$-nilpotent element for which $\co(x)$ is pointed. 
Then there exists a derivation 
$D \in D_{\rm can} + \ad \g$ with $Dx = x$ and 
\[ \Spec(D) \subeq  \big\{ 0, \pm \shalf, \pm 1\big\}.\]

Any invariant cone  $W$ generated by $W_\fl := W \cap \fl$ and 
a central cone $W_\fz \subeq \fz$ satisfies $e^{\R D}W = W$. 
\end{thm}

Recall that, by Corollary~\ref{cor:1.7}, 
any $\ad$-nilpotent element $x$ with $C_x$ pointed is 
conjugate to an element of $\fz + \fs$. 

\begin{prf} By Corollary~\ref{cor:1.7}, 
$x_\fs := x_\fl \in \fs = [\fl,\fl]$. 
The nilpotent elements $x_\fs \in \fs$ is contained in the sum of 
all non-compact (hence hermitian) simple ideals 
and $C_{x_\fs}$ is pointed  because the pointedness of 
\[ \co_\fs(x_\fs) \subeq \co(x_\fs) = \co(x) - x_\fz \] 
implies that $C_{x_\fs} \subeq \fs$ is pointed (Proposition~\ref{prop:3.14}). 
Now the classification of nilpotent elements in invariant cones 
implies the existence of an $\ad$-diagonalizable element $h_\fs \in \fs$ with 
\begin{equation}
  \label{eq:spechs}
[h_\fs,x_\fs] = x_\fs \quad \mbox{  and } \quad 
 \Spec(\ad h_\fs) = \big\{ 0, \pm \shalf, \pm 1\big\}
\end{equation}
(\cite[Lemma~IV.7]{HNO94}). 
The element $h_\fs$ is obtained by using a system of strongly orthogonal 
restricted roots to show that $x_\fs$ is contained in a subalgebra 
$\fb \cong \fsl_2(\R)^r$ in which $h_\fs$ is an Euler element 
(but in general not in $\fs$ or $\g$).
For any element $h_\fs \in \fs$ with $[h_\fs,x_\fs] = x_\fs$ 
(cf.\ Remark~\ref{rem:jac-mor}), the derivation 
\[ D := D_{\rm can} + \ad h_\fs \] 
(see \eqref{eq:DV1}) 
then satisfies 
\[ Dx  = D(x_\fz + x_\fs) = x_\fz + [h_\fs,x_\fs] = x_\fz + x_\fs = x.\] 
As $D_{\rm can} h_\fs = 0$, the derivation $D$ is diagonalizable 
because both summands commute and are diagonalizable. 
The eigenvalues on $\fs$ are contained in $\big\{ 0, \pm \shalf, \pm 1\big\}$ 
by \eqref{eq:spechs} and $\fz \subeq \ker(D - \1)$. 

Let 
\[ V = V_0(h_\fs) + V_{1/2}(h_\fs) + V_{-1/2}(h_\fs) \] 
denote the $h_\fs$-eigenspaces in $V$ 
(Lemma~\ref{lem:conv-module-euler-elmnt}). 
Then the corresponding eigenvalues of $D$ on $V$ are 
$\shalf, 1$ and $0$. This completes the proof of the first assertion.  

If the invariant cone $W \subeq \g$ 
is generated by $W_\fl$ and 
a central cone $W_\fz$, then the invariance of both cones 
under $e^{\R D_{\rm can}}$ implies that $e^{\R D_{\rm can}}W = W$, 
hence that $e^{\R D}W = W$ follows from $D \in D_{\rm can} + \ad\g$. 
\end{prf}

\begin{rem} \mlabel{rem:Euler-suppl}
If the ideal $\fs'$ of $\fs$ generated by a nilpotent 
element $x_\fs \in \fs$ contains only simple summands of tube type, then 
their restricted root systems are of type $(C_r)$ and never 
of type $(BC_r)$ (\cite[pp.~587-588]{HC56}). Therefore 
\cite[Lemma~IV.7]{HNO94} actually provides an 
Euler element $h_\fs$ of $\fs'$, and hence also of $\fs$. 
\end{rem}

\subsection{Euler derivations} 

\begin{defn} We call $D \in \der(\g)$ an {\it Euler derivation} 
if $D$ is diagonalizable with 
\[ \Spec(D) \subeq \{-1,0,1\}.\] 
\end{defn}

In this section we ask, for a nilpotent element $x$ 
for which $C_x$ is pointed for an Euler derivation $D$ satisfying $Dx = x$. 
Recall that the latter relation implies that $x$ 
is nilpotent (cf.\ Definition~\ref{def:4.1}).

\begin{rem} If $\g$ is simple hermitian, 
then every derivation of $\g$ is inner and 
an Euler derivation exists if and only if $\g$ is of tube type 
(\cite[Prop.~3.11(b)]{MN21}). 
Therefore Euler derivations $D$ with $Dx = x$ need not exist. 
Concrete examples of hermitian Lie algebras without 
Euler derivations are $\su_{p,q}(\C)$ for $p \not=q$. 
If, however, $\g$ is hermitian of tube type, then 
there exists an Euler element $h$ 
with $[h,x] = x$ (Remark~\ref{rem:Euler-suppl}). 
\end{rem}

\begin{rem} \mlabel{rem:struc-euler-der}
Let $\g = \g(\fl,V,\fz,\beta)$ be an admissible Lie algebra 
and $D'\in \der(\g)$ an Euler derivation. Then $D'$ is conjugate 
under $e^{\ad \fu} = e^{\ad V}$ to an Euler derivation $D$ such that 
\[ D(\fl) \subeq \fl, \quad D(\fz(\fl)) = \{0\} \] 
(\cite[Thm.~8.1.5]{Oeh21}). Then $D(\fz) \subeq \fz$ and $D(\fu) \subeq \fu$ 
imply that 
\[ D(V)= D([\fl,\fu]) 
\subeq [D\fl,\fu] + [\fl, D\fu] \subeq [\fl,\fu] = V.\] 
Hence there exists an element $h \in \fs := [\fl,\fl]$ and endomorphisms 
$D_\fz \in \End(\fz)$ and $D_V \in \End(V)$ with 
\[ D(z,v,x) = (D_\fz z, D_V v + [h,v], [h,x]) \quad \mbox{ for }\quad 
(z,v,x) \in \g.\]
It follows in particular that $D$ preserves all ideals of $\fl$. 
\end{rem}

\begin{prop}
  Let \(\g\) be an admissible non-reductive Lie algebra with 
$\fz(\g) \subeq [\g,\g]$. 
If $h$ is an Euler element of $\g$, then $[h,\fr] = \{0\}$. 
\end{prop}

\begin{proof}   By \cite[Thm.\ 8.1.5]{Oeh21}, we may assume that 
\(\g = \g(\fl, V, \fz, \beta)\) with $[h,\fl] \subeq \fl$. 
We show that this implies that $h \in \fz + \fs$ for $\fs = [\fl,\fl]$. 
Write $h = h_\fz + h_V + h_\fl$. 
Then $[h,\fl] \subeq \fl$ implies that 
\[ [h_V,\fl] \subeq \fl \cap [V,\fl] \subeq \fl \cap V = \{0\}.\] 
As $\fz_V(\fl) \subeq \fz_V(\ft_\fl)= \{0\}$ 
(Theorem~\ref{thm:spind}(d)), this implies $h_V = 0$, i.e., 
$h \in \fz + \fl$ and therefore $\ad h \in \ad(\fl)$. 
We may therefore assume that $h \in \fl$. 

Write $h= h_0 + h_1$ with $h_0 \in \fz(\fl)$ and $h_1 \in \fs$. 
These two summands commute and $h_1$ is an Euler element in $\fs$. 
As elements of $\fz(\fl)$ acts with purely imaginary spectrum on $V$, 
we further obtain $h_0 = 0$, so that $h \in \fs$. 

As the radical $\fr$ of \(\g\) is 
\[ \fr = \fz + V + \fz(\fl) = [V,V] + V + \fz(\fl)\] 
and $[h,\fz(\fl)] = \{0\}$, it remains to show that $[h,V] = \{0\}$. 
  Let $f \in \fz^*$ be 
such that \((V,f \circ \beta)\) is a symplectic \(\fl\)-module of convex type.
  By applying Lemma \ref{lem:conv-module-euler-elmnt} to the ideal \(\fs' \subset \fl\) generated by \(\fs_{\pm 1}(h)\) and the \(\fl\)-module 
\((V_{\fs'}, f \circ \beta)\), 
we see that \(2\ad(h)\) acts by an involution on $V_{\fs'}$. 
As $h$ is an Euler element of $\g$, it follows that 
$V_{\fs'} = \{0\}$, hence that $[\fs',V] = \{0\}$, and 
thus in particular $[h,V] = \{0\}$. 
\end{proof}

\begin{ex} Let $(V,\omega)$ be a finite dimensional symplectic vector space. 
We consider $\g = \g(\sp(V,\omega), V, \R, \omega)$ 
Then the Levi complement is 
$\sp(V,\omega)$, and all Euler elements $h$ in $\sp(V,\omega)$ are conjugate to 
$\shalf \tau$, where $\tau$ is an antisymplectic involution on $V$ 
(\cite[Prop.~3.11(b)]{MN21}).
Therefore the eigenvalues of $h$ on $\g$ contain $\pm \shalf$ 
and both Lie algebras contain no Euler element. 
\end{ex}

The following result shows in particular 
that solvable admissible Lie algebras 
do not even possess Euler derivations. 

\begin{prop}
  Let \(\g\) be a solvable admissible Lie algebra with \(\fz(\g) \subset [\g,\g]\). Then there exists no non-zero Euler derivation on \(\g\).
\end{prop}

\begin{proof}
  Let \(D \in \der(\g)\) be an Euler derivation. 
By \cite[Thm.\ 8.1.5]{Oeh21}, we may assume that 
\(\g = \g(\fl, V, \fz, \beta)\) as in Theorem~\ref{thm:spind} 
with $D(\fl) \subeq \fl$. 
As $\g$ is solvable, the Lie algebra $\fl$ is abelian, 
so that \cite[Thm.\ 2.27]{Oeh20b} shows that \(D(\fl) = \{0\}\). 
In particular, \(V = [\fl, \fu]\) is \(D\)-invariant, so that
  \[D(z,v,x) = (D_\fz z, D_V v, 0) \quad \text{for } (z,v,x) \in 
\fz \times V \times \fl = \g,\]
  where \(D_V \in \End_\fl(V)\) is 
diagonalizable. Since \(\g\) is admissible, there exists \(f \in \fz^\star\) such that \((V, f \circ \beta)\) is a symplectic \(\fl\)-module of convex type.
  Hence, every $D_V$-eigenspace $V_\lambda(D_V) \subeq V$  
defines a symplectic \(\fl\)-module 
\((V_\lambda(D_V), f \circ \beta)\) 
of convex type. In particular we have 
\(\beta(V_\lambda(D_V), V_\lambda(D_V)) \neq \{0\}\) and, since \(D\) is a derivation,
  \[\{0\} \neq \beta(V_\lambda(D_V), V_\lambda(D_V)) = [V_\lambda(D_V),V_\lambda(D_V)] \subset \fz_{2\lambda}(D_\fz).\]
  Since \(\spec(D) \subset \{-1,0,1\}\), this implies \(\lambda = 0\) and 
therefore \(D_V = 0\). As $\fz(\g) = [V,V]$, it follows that  $D = 0$.
\end{proof}

Let \(\g = \g(\fl,V,\fz,\beta)\) be admissible. For \(h \in \fl\), we recall 
the subspaces 
\[V_h = [h,V] \quad \text{and} \quad V_{h,0} = \{v \in V : h.v = 0\} = V_h^{\bot_\beta}\]
(cf.\ Lemma~\ref{lem:v0v1}).

\begin{lem} \mlabel{lem:eulder}
Let $\g = \g(\fl,V,\fz,\beta)$ be admissible with 
$\fz \subeq [V,V]$ and $h \in [\fl,\fl]$ an Euler element. Then there exists 
a non-zero Euler derivation $D$ on $\g$ with $D(\fl) \subeq \fl$ and 
$D\res_{\fl} = \ad_\fl h$ if and only if 
there exists an $\fl$-module 
decomposition $V_h = V_h^+ \oplus V_h^-$ such that
the sum of the three subspaces 
$[V_h^+,V_h^+], [V_h^-,V_h^-]$,  and  \([V_h^+,V_h^-] + [V_{h,0},V_{h,0}]\) of 
$\fz$ is direct. 
\end{lem}

\begin{prf} If $D$ exists, then the invariance of 
$\fl$ implies $D(V) \subeq V$ (Remark~\ref{rem:struc-euler-der}) 
Therefore $D$ is of the form 
\begin{equation}
  \label{eq:deriv}
D(z,v,y) := (D_\fz z, [h,v] + D_V v, [h,y])
\quad \mbox{ for }  \quad (z,v,y) \in \fz \times V \times \fl,
\end{equation}
where $D_V \in \End_\fl(V)$. Further, the Classification Theorem 
\cite[Thm.\ 3.14]{Oeh20b} implies that 
 \(\ker(D_V) = V_{h,0}\), and that \(2h\) acts as 
an involution on 
\[ V_h = V_{1/2}(D_V) + V_{-1/2}(D_V).\] 
  We set \(V_h^{\pm} := V_{\pm 1/2}(D_V)\).
Then the assertion follows from 
\[[V_{h,0},V_{h,0}] + [V_h^+,V_h^-] 
= \fz_0(D_\fz) \quad \text{and} \quad [V_h^\pm,V_h^\pm] = \fz_{\pm 1}(D_\fz).\]

Conversely, suppose that the sum of 
$[V_h^+,V_h^+], [V_h^-,V_h^-]$  and  \([V_h^+,V_h^-] + [V_{h,0},V_{h,0}]\) 
is direct. Since \(h\) is an Euler element, the subalgebra 
$\fs_h \subeq \fl$ generated by $\fl_{\pm 1}(h)$ 
is a semisimple ideal of \(\fl\) 
which is a direct sum of hermitian simple ideals of tube type 
(\cite[Prop.~3.11(b)]{MN21}).
By Lemma \ref{lem:conv-module-euler-elmnt}, the element \(2h\) acts on 
\(V_{{\rm eff},\fs_h} = V_h\) as an involution. We define
\[D_V := \big(\tfrac{1}{2}\id_{V_h^+}\big) \oplus 
\big(- \tfrac{1}{2}\id_{V_h^-}\big) \oplus 0\cdot \id_{V_{h,0}}.\]
Then \(D_V\) commutes with the action of \(\fl\) on \(V\) because \(V_h^\pm\) and \(V_{h,0}\) are $\fl$-invariant ($\fl = \fs_h \oplus \fz_\fl(\fs_h)$).
Our assumption now implies that we can define an endomorphism \(D_\fz \in \End(\fz)\) by 
\[ \fz_{\pm 1}(D_\fz) = [V_h^\pm,V_h^\pm] 
\quad \mbox{ and } \quad 
\fz_0(D_\fz) = [V_{h,0},V_{h,0}] + [V_h^+,V_h^-].\] 
Then \eqref{eq:deriv} 
defines an Euler derivation of \(\g\) with 
$D(\fl) \subeq \fl$ and $D\res_\fl = \ad h$ 
(cf.\ \cite[Thm.\ 3.14]{Oeh20b}). 
\end{prf}

\begin{rem} \mlabel{rem:euler3}
(a) To see the subalgebra $\g_D \subeq \g$ 
generated by $\g_{\pm 1}(D)$ in 
the context of the preceding theorem, we first observe that 
\[ \g_0(D) = \fz_0(D_\fz) 
\oplus  V_{h,-1/2}^+(h)
\oplus  V_{h,1/2}^-(h)
\oplus  V_{h,0} \oplus  \fl_0(h).\] 
For the ideal $\fs_h \trile \fl$ generated by  $h$, 
we have $\g_D \cap \fl = \fs_h$. 
Next we observe that 
\[ [\fl_{\pm 1}(h), V^+_{h,\mp 1/2}(h)] = V^+_{h,\pm 1/2}(h).\] 
If $\g = \g_D$, then we must have  $\fl = \fs_h$, 
and then $V = [\fl,V]$ implies $V_{h,0} = \{0\}$. Conversely, 
$\fl = \fs_h$ and $V_{h,0} = \{0\}$ imply $\g = \g_D$ 
because in this case $\fz_0(D_\fz) = [V_h^+, V_h^-]$. 

\nin (b) If $\fl = \fs_h$ and, as a consequence, 
$V_{h,0} = \{0\}$, then we may put $V_h^+ := V$, which leads 
to  $D_V = \frac{1}{2} \id_V$  and $D_\fz = \id_\fz$, so that
$D_\fz \oplus D_V = D_{\rm  can}$. 
Thus $D = D_{\rm can} + \ad h$ is an Euler derivation of $\g$ 
with $\g_0(D) = [\g_1(D), \g_{-1}(D)]$. 
\end{rem}

\begin{thm} \mlabel{thm:eulderexist}{\rm(Existence of Euler affine pairs)} 
  Let \(\g = \g(\fl,V,\fz,\beta)\) be an admissible Lie algebra with 
$\fz = \fz(\g)\subeq [\g,\g]$ 
and let $x = x_\fz +x_V + x_\fs \in \g$ be a nilpotent 
element with $C_x$ pointed. 
Then there exists an Euler derivation \(D \in \der(\g)\) with \(Dx = x\) and 
$D(\fl) \subeq \fl$ 
if and only if there exists an Euler element $h$ of \(\fl\) 
and a decomposition $V_h = V_h^+ \oplus V_h^-$ into $\fl$-submodules 
such that
      \begin{enumerate}
        \item[\rm(a)] \([h,x_\fs] = x_\fs\),
        \item[\rm(b)] \(x_\fz \in [V_h^+, V_h^+]\), $x_V \in V_h^+$, and
        \item[\rm(c)] The sum of the three subspaces 
$[V_h^+,V_h^+], [V_h^-,V_h^-]$,  and  \([V_h^+,V_h^-] + [V_{h,0},V_{h,0}]\) of 
$\fz$ is direct. 
      \end{enumerate}
\end{thm}


\begin{proof} 
Suppose first that an Euler derivation \(D\) with \(Dx = x\) and 
$D(\fl) \subeq \fl$ exists.
  Then, by \cite[Thm.\ 3.14]{Oeh20b} (here we use that 
$\fz= \fz(\g) \subeq [\g,\g]$),  
we have $D(\fz(\fl)) = \{0\}$  and there exists an 
Euler element $h \in [\fl,\fl]$ such that 
  \[D(z,v,y) := (D_\fz z, [h,v] + D_V v, [h,y])
\quad \mbox{ for }  \quad (z,v,y) \in \fz \times V \times \fl.\]
By Lemma~\ref{lem:eulder} and its proof, 
the $\fl$-submodules \(V_h^{\pm} := V_{\pm 1/2}(D_V)\) 
satisfy (c). Further, (a) is the $\fl$-component of $Dx = x$, 
and the first part of (b) follows from
  \[x_\fz \in \fz_1(D_\fz) = [V_{1/2}(D_V),V_{1/2}(D_V)] = [V_h^+,V_h^+] \]
(cf.\ \cite[Thm.\ 3.14]{Oeh20b}). 
For the second part, we recall from the proof of the 
Reduction Theorem~\ref{thm:conj1} that 
$x_V \in [x_\fs,V]$. 
As $[x_\fs, V_\lambda(h)] \subeq V_{\lambda +1}(h)$ by (a), we have 
\[ [x_\fs,V] 
= [x_\fs, V_{-1/2}(h)+ V_0(h) + V_{1/2}(h)] 
= [x_\fs, V_{-1/2}(h)] 
\subeq V_{1/2}(h),\] 
hence in particular $x_V \in V_{1/2}(h)$. 
Eventually, $Dx = x$ entails 
$x_V \in V_{1/2}(D_V) = V_h^+$. 

Conversely, suppose that (a)-(c)  hold 
and construct the Euler derivation as in 
Lemma~\ref{lem:eulder} such that 
\[D_V := \big(\tfrac{1}{2}\id_{V_h^+}\big) \oplus 
\big(- \tfrac{1}{2}\id_{V_h^-}\big) \oplus 0\id_{V_{h,0}}
\quad \mbox{ and } \quad 
\fz_{\pm 1}(D_\fz) = [V_h^\pm,V_h^\pm].\] 
Then $D_\fz x_\fz = x_\fz$ by the first part of (b),
and the second part of  (b), combined with 
$x_V \in D_{1/2}(h)$ implies that 
\[ Dx_V = D_V x_V +[h,x_V] = \frac{1}{2} x_V +  \frac{1}{2} x_V = x_V.\]
Thus $Dx = x$ follows from (a).
\end{proof}

Here is an example of a rather small admissible Lie algebra that 
already displays the complexity of the situation we encounter in 
Theorem~\ref{thm:eulderexist}. 

\begin{ex} For the Lie algebra $\fl := \R z\oplus \fsl_2(\R) 
\cong \gl_2(\R)$, we consider the representation $\rho$ on 
\[ V := V_1 \oplus V_2 \oplus V_3, \quad V_j \cong \R^2,\] 
given by 
\[  \rho(z) = 0 \oplus 0 \oplus \pmat{0 & 1 \\ -1 & 0}, \quad 
\rho(x) := x \oplus x \oplus 0 \quad \mbox{ for }\quad  x \in \fsl_2(\R).\]
We then have $3$ invariant alternating forms on $V$, represented by 
\[ \beta_+(\bv,\bw) 
= v_1 w_2 - v_2 w_1, \quad 
\beta_-(\bv,\bw) 
= v_3 w_4 - v_4 w_3, \] 
and 
\[ \beta_0(\bv,\bw) 
= v_1 w_4 - v_2 w_3 
+ v_3 w_2 - v_4 w_1 + v_5 w_6 - v_6 w_5.\] 
For $x := \pmat{0 & 1 \\ -1 & 0}$, the corresponding Hamiltonians are 
\[ \beta_+(x.\bv, \bv) 
= v_1^2 + v_2^2, \quad 
\beta_-(x.\bv, \bv) 
= v_3^2 + v_4^2,  
\quad \mbox{ and } \quad 
 \beta_0(x.\bv, \bv) 
= v_1 v_4 + v_2 v_3 + v_4 v_2 + v_3 v_1.\]
Further, $\beta_0(z.\bv,\bv) = v_5^2 + v_6^2.$ 
Now 
$\beta = (\beta_+,\beta_-,\beta_0)$ defines an invariant alternating 
form with values in $\fz := \R^3$. 
For $f(x) = 2 x_1 + 2x_2 + x_3$ and $\omega := f \circ \beta$, the Hamiltonian 
function $H_x$ is positive definite. 

For the Euler element $h := \frac{1}{2}\diag(1,-1) \in \fsl_2(\R)$, 
we have 
\[ V_h = V_1 + V_2 \quad \mbox{ and }  \quad V_{h,0} = V_3.\] 
We define a derivation $D_\fz + D_V$ on $\heis(V,\beta)$ by 
$D_V = \diag(1/2,-1/2,0) \otimes \id_{\R^2}$ and $D_\fz = \diag(1,-1,0).$
Then 
\[ \beta(V_h,V_h) \not\subeq \fz_1(D_\fz) + \fz_{-1}(D_\fz) \] 
because $\beta_0$ does not vanish on $V_h \times V_h$. 
\end{ex}

If $D$ is an Euler derivation of 
$\g = \g(\fl,V,\fz,\beta)$ with 
$\g_0(D) = [\g_1(D), \g_{-1}(D)],$ then it induces 
on $\fl \cong \fg/\fu$ an Euler derivation with 
$\fl_0(D) = [\fl_1(D), \fl_{-1}(D)]$, and this implies 
in particular that $\fl$ is semisimple and a direct sum 
of simple hermitian ideals of tube type (\cite[Prop.~3.11(b)]{MN21}).  
This observation justifies the assumption in our final proposition.

\begin{prop}
  Let \(\g = \g(\fs,V,\fz,\beta)\) be an admissible Lie algebra, where \(\fs\) is a direct sum of hermitian simple Lie algebras of tube type 
and $\fz = [V,V]$. 
Let \(x \in \g\) be a nilpotent element with $C_x$ pointed. Then there exists an Euler derivation \(D \in \der(\g)\) such that 
\begin{equation}
  \label{eq:fineq}
 Dx = x\quad \mbox{  and  } \quad \g_0(D) = [\g_1(D), \g_{-1}(D)].
\end{equation}
\end{prop}

\begin{proof} 
By Corollary~\ref{cor:1.7}, we may assume that 
$x = x_\fz + x_\fs \in \fz + \fs$. 
Then \(x_\fs\) is nilpotent with $C_{x_\fs}$ pointed 
(cf.\ Corollary~\ref{cor:2.6}). 
  Hence, there exists an Euler element \(h \in \fs\) such that \([h,x_\fs] = x_\fs\) and such that \(\fs\) is generated by \(\fs_{\pm 1}(h)\) 
(\cite[Lemma~IV.7]{HNO94}). 
Now the discussion in Remark~\ref{rem:euler3}(b) shows that 
$D := D_{\rm can} + \ad h$ is an Euler derivation 
satisfying~\eqref{eq:fineq}. 
\end{proof}

\appendix

\section{Tools concerning convexity} 
\mlabel{app:a}

\begin{lem} \mlabel{lem:limcone} {\rm(\cite[Lemma~2.9]{Ne10})} 
If $\eset\not=C \subeq E$ 
is an open or closed convex subset, then the following assertions hold: 
\begin{description}
  \item[\rm(i)] $\lim(C) := \{ x \in E \: x + C \subeq C\}$ 
is a closed convex cone that coincides with 
$\lim(\oline C)$. 
  \item[\rm(ii)] $v \in \lim(C)$ if and only if 
there exist net $t_j c_j \to v$, where 
$t_j \geq 0$, $t_j \to 0$ and $c_j \in C$. 
  \item[\rm(iii)] If $c \in C$ and $d \in E$ satisfy 
$c + \R_+ d \subeq C$, then $d \in \lim(C)$. 
\item[\rm(iv)] $H(C) := \lim(C) \cap -\lim(C)$ is trivial 
if and only if $C$ contains no affine lines. 
\end{description}
\end{lem}

\begin{lem}
  \mlabel{lem:conecrit} 
Let $V$ be a finite dimensional real vector space and 
$C \subeq V$ be a closed convex subset. 
Then the cone $\cone(C):= \oline{\R_+ C}$ is pointed if  
$\lim(C)$ is pointed and $0 \not\in C$. 
\end{lem}

\begin{prf} First we show that 
  \begin{equation}
    \label{eq:inc1}
 \cone(C) \subeq S := \R_+ C \cup \lim(C).
  \end{equation}
Let $v \in \cone(C)$ and choose $\lambda_n \geq 0$ and 
$c_n \in C$ with $\lambda_n c_n \to v$. 

As $0 \not\in C$, the Hahn--Banach Separation Theorem implies the existence 
of an $f \in V^*$ such that 
$0 < \inf f(C) =: \delta$. Then 
$\lambda_n \delta \leq \lambda_n f(c_n) \to f(v)$ 
shows that the sequence $\lambda_n$ is bounded. 
We may therefore assume that $\lambda_n \to \lambda \in [0,\infty)$. 
If $\lambda > 0$, then the sequence $c_n$ converges to $\lambda^{-1} v$ 
and $v \in \R_+ C$. If $\lambda = 0$, then $v \in \lim(C)$ 
(Lemma~\ref{lem:limcone}), and \eqref{eq:inc1} follows.

Next we observe that $C + \lim(C)\subeq C$ implies 
$\R_+ C + \lim(C) \subeq \R_+ C$. Therefore 
$S$ is an additive subsemigroup of~$V$ 
in which $\R_+ C$ is a semigroup ideal, i.e., 
$S + \R_+ C \subeq \R_+C$. Since $0\not\in \R_+C$, it follows that 
\[ S \cap - S \subeq \lim(C) \cap - \lim(C) = \{0\}.\] 
This implies that $\cone(C) \subeq S$ is pointed. 
\end{prf}

\begin{lem} \mlabel{lem:pointconegen} 
Let $C \subeq V$ be a pointed closed convex cone and $x \in V$. 
Then the following are equivalent: 
\begin{itemize}
\item[\rm(a)] $\cone(x + C) = \oline{\R_+ x + C}$ is pointed. 
\item[\rm(b)] $(x + C) \cap -C \subeq \{0\}$.
\item[\rm(c)] $0 \not\in x + C$ or $x = 0$. 
\end{itemize}
\end{lem}

\begin{prf} (a) $\Rarrow$ (b): Suppose the $D := \cone(x + C)$ is pointed. 
Then $C \subeq D$ and 
\[ (x + C) \cap - C \subeq  D \cap - D = \{0\}.\] 

\nin (b) $\Rarrow$ (c): That the intersection of $x+ C$ and $-C$ is contained 
in $\{0\}$ happens in two cases. 
If the intersection is empty, then $0 \in - C$ shows that $0$ is not contained in 
$x + C$. If the intersection is non-empty, then it is $\{0\}$ and thus 
$x \in - C$, which in turn implies $x = 0$. 


\nin (c) $\Rarrow$ (a): If $0 \not\in x  + C$, then Lemma~\ref{lem:conecrit} 
implies that $\cone(x + C)$ is pointed. If $x = 0$, then 
$\cone(x + C) = \cone(C) = C$ is trivially pointed. 
\end{prf}

\section{Tools concerning Lie algebras} 
\mlabel{app:b}

\begin{prop} \mlabel{prop:jac-mor}  
Let $x$ be an element of the semisimple real Lie algebra 
$\g$ and $x = x_s + x_n$ its Jordan decomposition, where $x_s$ is semisimple 
and $x_n\not=0$ is nilpotent. Then there exists a reductive subalgebra 
$\fm \subeq \g$ such that 
\[ \fm \cong \fsl_2(\R) \quad \mbox{ if } \quad x_s = 0\] 
and 
\[ \fm \cong \gl_2(\R) \quad \mbox{ with } \quad 
x_s \in \fz(\fm), x_n \in [\fm,\fm] 
\quad \mbox{ if } \quad x_s \not= 0.\] 
\end{prop}

\begin{prf} Since the Jordan decomposition and the adjoint orbit 
of $x$ adapts to the decomposition of $\g$ into simple ideals, 
we may w.l.o.g.\ assume that $\g$ is simple. 

Let $\fq = \fl \ltimes \fu\subeq \g$ denote the Jacobson--Morozov 
parabolic associated to the nilpotent element $x_n$ (\cite{HNO94}). Then 
$x_s \in \ker(\ad x_n) \subeq \fq$ implies that $x_s \in \fq$. 
As $x_s$ is semisimple, it is conjugate under the group 
of inner automorphisms of 
$\fq$ to an element of $\fl$.\begin{footnote}{Every algebraic subgroup 
$G \subeq \GL(V)$, $V$ a finite dimensional real vector space,
 is a semidirect product 
$G \cong U \rtimes L$, where $U$ is unipotent and $L$ is reductive. 
Moreover, for every reductive  subgroup $L_1 \subeq G$ there exists an 
element $g \in G$ with $gL_1 g^{-1} \subeq L$ 
(\cite[Thm.~VIII.4.3]{Ho81}).}
\end{footnote}
By the Jacobson--Morozov Theorem (\cite[Ch.~VIII, \S 11, Prop.~2]{Bo90}), 
$\fl$ contains a semisimple element $h$ with 
$[h,x_n] = 2 x_n$ and $h \in [x_n,\g]$. In terms of this element, we have 
$\fq = \sum_{n \geq 0} \g_n(h)$ and $\fl = \ker(\ad h)$. 
We further find a nilpotent element $y \in \g_{-2}(h)$ such that 
$[x_n,y] = h$, so that the Lie algebra generated by 
$x$ and $y$ is isomorphic to $\fsl_2(\R)$. 
Replacing $x$ by a suitable conjugate, we have seen above that 
we may assume that $x_s \in \g_0(h)$. We consider the Lie algebra 
\[ \fm := \Spann \{ x_s, h,x_n,y\}.\] 
If $x_s = 0$, then $\fm \cong \fsl_2(\R)$. 
If $x_s \not=0$, then 
\[ 0 = [h,x_s] = [[x_n, y],x_s] = [x_n, [y,x_s]] 
\quad \mbox{ with } \quad [y,x_s] \in \g_{-2}(h),\] 
so that the representation theory of $\fsl_2(\R)$ implies that 
$[y,x_s] = 0$. Therefore $x_s \in \fz(\fm)$. 
\end{prf}

\begin{cor} \mlabel{cor:jac-mor}  
Let $x$ be an element of the semisimple real Lie algebra 
$\g$ and $x = x_s + x_n$ its Jordan decomposition, where $x_s$ is semisimple 
and $x_n\not=0$ is nilpotent. Then the adjoint orbit $\cO_x$ of $x$ contains 
all elements of the form $x_s + t x_n$, $t > 0$. 
In particular, 
\[ x_s \in \co(x) \quad \mbox{ and } \quad x_n \in \lim(\co(x)).\] 
\end{cor}

\begin{prf} With Proposition~\ref{prop:jac-mor} we find an 
element $h \in \g$ with $[h,x_s] = 0$ and $[h,x_n] = 2x_n$. 
Then the assertion follows from 
$e^{t\ad h} x = x_s + e^{2t} x_n$ for $t \in \R$.  
\end{prf}

\end{document}